\documentclass[11pt]{article}
\usepackage{latexsym,amsmath,amssymb,amsfonts,mathrsfs,amsthm}
\usepackage{epsf,graphicx,epsfig,color,cite,cases}
\usepackage{subfigure,graphics,multirow,marginnote,enumerate,bm}
\usepackage{indentfirst}
\newcommand{\be}{\begin{eqnarray}}
\newcommand{\ben}{\begin{eqnarray*}}
\newcommand{\en}{\end{eqnarray}}
\newcommand{\enn}{\end{eqnarray*}}

\newtheorem{theorem}{Theorem}[section]
\newtheorem{lemma}[theorem]{Lemma}
\newtheorem{corollary}[theorem]{Corollary}

\newtheorem{assumption}{Assumption}
\newtheorem{remark}[theorem]{Remark}
\newtheorem{proposition}[theorem]{Proposition}

\begin{document}
\renewcommand{\theequation}{\arabic{section}.\arabic{equation}}

\title{\bf Elastic scattering problems by penetrable obstacles with embedded objects}

 \author{Chun Liu\thanks{School of Mathematical Sciences and LPMC, Nankai University, Tianjing 300071, China({\tt liuchun@nankai.edu.cn})}
 \and
 Jiaqing Yang\thanks{School of Mathematics and Statistics, Xi'an Jiaotong University,
 Xi'an, Shaanxi 710049, China ({\tt jiaq.yang@mail.xjtu.edu.cn})}
\and
Bo Zhang\thanks{LSEC, NCMIS and Academy of Mathematics and Systems Science, Chinese Academy
of Sciences, Beijing 100190, China and School of Mathematical Sciences, University of Chinese
Academy of Sciences, Beijing 100049, China ({\tt b.zhang@amt.ac.cn})}
}
\date{}

\maketitle


\begin{abstract}
This paper considers 3-D elastic scattering problems by penetrable obstacles with embedded objects. The well-posedness of transmission
 problem is proved by employing integral equation method. Then the Inverse Problems  , which is to recover the obstacle
 by the far-field pattern measurement, is considered. It is shown that the inhomogeneous penetrable obstacle can be
 uniquely determined from the far-field pattern at a fixed frequency.

\vspace{.2in}


\end{abstract}

\maketitle

\section{Introduction}

Elastic wave theory has a variety of applications in geophysics, nondestructive detection and seismology. The inverse scattering problem is to reconstruct the physical properties of elastic medium or to detect cracks in solids via the collection of wave fields, such as nondestructive evaluation of concrete structures \cite{JS1958,JH1989}, earthquake prediction and oil exploration \cite{JP1977,HP2008}.  

In this work we place the obstacle in a three-dimensional homogeneous and isotropic elastic
medium and we assume that it is penetrable (a so-called inclusion ) with some embedded objects. We consider as incident wave an elastic longitudinal or transversal wave that after interacting with the boundary of the medium is split into an interior and a scattered wave, propagating in the inclusion and the exterior, respectively. The scattered wave is also decomposed into a longitudinal and a transversal wave with different wave numbers which  makes the study of elastic wave more difficult. 

Before considering the Inverse Problems , we should have a good knowledge of the  \emph{direct problem}, which is to find the scattered field and its far-field patterns from the knowledge of the obstacle and the incident wave. In the case that there are no embedded objects, the direct problem is linear and well posed for smooth obstacles in 2-D\cite{MARTIN1990ON}, in 3-D \cite{Charalambopoulos2002On,Charalambopoulos2007On} by the boundary integral method or the variational method. If there exists embedded objects, the direct problem is also well-posed in acoustics and electromagnetics \cite{Liu2010Direct, Liu2010The,Yang2017Uniqueness}. The  \emph{Inverse Problems } is to find the shape and the position of the inclusion from the knowledge of the far-field patterns.

In the case that there are no embedded objects, many uniqueness results have been obtained in determining the inclusion in acoustics and electromagnetics. The first uniqueness result was established by Isakov\cite{Isakov1990On} in 1990, and the idea is to construct singular solutions of the scattering problem with respect to two different inclusions with same far-field patterns. Then \cite{KK1993} simplified Isakov's method by using the integral method to establish a priori estimate of the solution on interface and it was also extended to the case of impenetrable obstacles. Since then, the idea has been applied to establish uniqueness results for many other inverse
scattering problems\cite{Isakov2008On,Liu2010Direct, Liu2010The,Liu2010uni,Liu2009A,Liu2012in}(and the references therein). In elasticity, there are many uniqueness results about impenetrable obstacle, we refer to \cite{peter1993uniqueness,Elschner2010,Xiao2019,Xiao2017,Gintides2012,Sini2015} for an incomplete list. But for penetrable obstacles, there are little results as far as we know. \cite{Peter1993A,Peter2002} show the uniqueness of the density if the inclusion and the background medium are the same. \cite{JG2021} proved that the far-field patterns for all incident plane waves can uniquely determine the penetrable obstacle based on the mixed reciprocity relations.

In the case that there exists embedded objects, \cite{Liu2010Direct} proved the uniqueness determination of inclusion in acoustic with a known refractive index and \cite{Liu2010The} got the same result for electromagnetic waves with a known refractive index. But in \cite{Yang2017Uniqueness}, it gave a new method to establish uniqueness results for determining inclusion from the knowledge of the acoustic or electric far-field measurement. It's based on constructing a well-posed interior transmission problem in a small domain inside inclusion associated with Helmholtz or Maxwell equations. In this paper, inspired by idea in \cite{Yang2017Uniqueness}, we consider the inverse transmission problem for elastic waves. We use boundary integral equation method to get the well-posedness of the direct problem. Different form acoustics and electromagetics, integral kernels in elasticity are hyper-singular, so we introduce the  \emph{symbol} of integral system to overcome the hyper singularity to obtain the well-posedness. Moreover, we get a  $L^p(4/3<p<2)$ estimate of the solution. Then for Inverse Problems , we use the 
method in \cite{Yang2017Uniqueness} by choosing a proper incident point source.

The paper is organized as follows: In Section \ref{sec1}, we formulate the problem in three dimensions and in Section \ref{sec2} we present the direct scattering problem, the elastic potential and the equivalent system of integral equations. The Inverse Problems  is stated in Section \ref{sec3} where we introduce the  \emph{interior transmission problem} and the \emph{modified interior transmission problem}, and prove some important results which are essential for the proof of Inverse Problems  . 

\section{Problem formulation}\label{sec1}
Consider the scattering of a time-harmonic elastic plane wave $ u^{inc} $ by an inhomogeneous penetrable
obstacle which may contain some embedded objects.
\begin{figure}[htbp]
\centering
\includegraphics[scale=0.5]{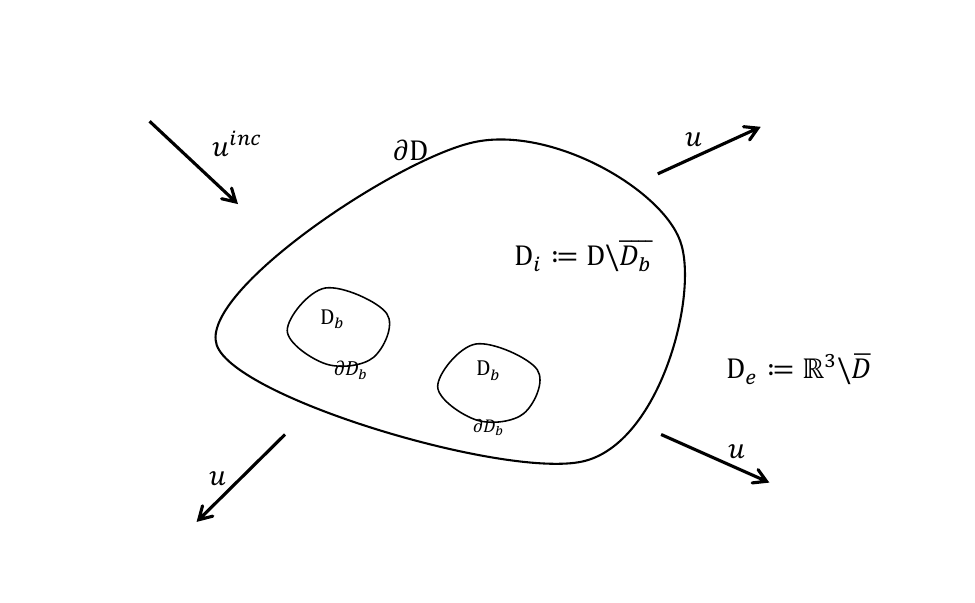}
\caption{ Scattering by inhomogeneous medium with embedded objects.}\label{f1}
\end{figure}

Here, $D\subset \mathbb{R}^3$ is a bounded open set with a $ C^2-  $boundary $\partial D$ and $D_b$ are bounded domains which are separate from each other. Let $D_i:= D\setminus \overline{D_b}$ be the inhomogeneous medium with Lam$\acute{e}$ constants $\lambda_i$, $\mu_i$ and density $\rho(x)\in  \mathbb{R}_{+} \cap L^{\infty}( D_i)$. Let $D_e:=\mathbb{R}^3 \setminus \overline{D}$ be the homogeneous medium with Lam$\acute{e}$ constants $\lambda_e$, $\mu_e$ and density $\rho_e$, which without loss of generality can be taken equal to one. In addition, parameters $\lambda_\alpha,\mu_\alpha \ (\alpha=i,e)$ satisfy $\mu_\alpha>0$ and $2\mu_\alpha+3\lambda_\alpha>0$.
 
Then, the propagation of $ u^{inc} $ can be described by following Navier equations
\begin{align}
\label{eq1} \mu_i\Delta v(x)+(\lambda_i+\mu_i)\nabla(\nabla\cdot v(x))+ \rho(x)\omega^2v(x)=&0, &&x\in D_i,\\
\label{eq2} \mu_e\Delta u(x)+(\lambda_e+\mu_e)\nabla(\nabla\cdot u(x))+ \omega^2 u(x)=&0,  &&x\in D_e,
\end{align} 
where $ v(x)  $ is the stationary field and $u(x)  $ is the scattered field, $\omega>0 $ is the angular  frequency. To assure
the continuity of displacement and stress through the surface $\partial D$, we need the following interfacial conditions
\begin{align}
\label{bd1} v(x)-u(x)= f(x), \qquad T_i v(x)-T_e u(x)=h(x), \qquad x\in \partial D , 
\end{align}
where  $f(x)=u^{inc}(x)$ and $h(x)=T_e u^{inc}(x)$.  We mention here that
 $T_\alpha (\alpha=i,e)$ is the surface stress (or traction) operator on the interior or exterior of the
 boundary $\partial D$
\begin{equation}\label{surface operator}
  T_\alpha = 2 \mu_\alpha \textbf{n} \cdot \text{grad}+\lambda_\alpha \textbf{n}\text{div} + \mu_\alpha \textbf{n} \times \text{curl} \qquad (\alpha=i,e),
\end{equation}
where $ \textbf{n}=(n_1,n_2,n_3)$ is unit normal vector on the boundary $\partial D$ at the point $x=(x_1,x_2,x_3)$ directed into $D_e$. For rigid bodies $D_b$, $ v(x) $ satisfies the first kind (Dirichlet) boundary condition
\begin{align}
 \label{bd3} v(x)= 0, \qquad x\in\partial D_b.
\end{align} 
The scattered field $u(x)$ can be decomposed into the sum $u(x)=  u_p(x)+ u_s(x)$, where 
\begin{equation}
 \nonumber
 u_p=-(k_e^p)^{-2}\nabla (\nabla\cdot u),\ \ u_s=-(k_e^s)^{-2}\nabla\times(\nabla\times u).
 \end{equation}
Here $u_p$ and $u_s$ are referred to as the longitudinal and transversal parts of the field $u$ respectively, satisfying the vector Helmholtz equation. Moreover, the wave numbers $k_e^p,  k_e^s>0$ are respectively given by
\begin{equation}
\nonumber
  k_e^p:=\frac{\omega}{\sqrt{2\mu_e+\lambda_e}},\ \ k_e^s:=\frac{\omega }{\sqrt{\mu_e}}.
\end{equation}
In this paper the incident wave is a plane pressure wave taking the form
\begin{equation}
\nonumber
u^{inc}(x)=d e^{ik_e^p d\cdot x},\qquad d\in  \mathbb{S}^2
\end{equation}
with the incident direction $d$. At last, the scattered field $u$ is required to satisfy the Kupradze radiation condition \cite{Kupradze}
\begin{align}
\label{rd} \lim_{r=|x|\rightarrow \infty }r\left[\frac{\partial u_\alpha(x)}{\partial r}-i k_e^\alpha u_\alpha(x)\right]=0,\qquad \alpha=p,s,
\end{align}
uniformly in all directions $\hat{x}=x/{|x|}\in \mathbb{S}^2$. The radiation conditions in (\ref{rd}) lead to the longitudinal part $u_p^{\infty}$ and the transversal part $u_s^{\infty}$ of the far-field pattern of $u$, given by the asymptotic behavior
\begin{align}
\label{far} u(x)=\frac{e^{i k_e^p|x|}}{4\pi(\lambda_e+\mu_e)|x|}u_p^{\infty}(\hat{x})+\frac{e^{i k_e^s|x|}}{4\pi\mu_e|x|}u_s^{\infty}(\hat{x})+\mathit{O}(\frac{1}{|x|^2}),\qquad |x|\rightarrow \infty,
\end{align}
where $ u_p^{\infty}(\hat{x}) $ is normal to $ \mathbb{S}^2 $ and $ u_s^{\infty}(\hat{x}) $ is tangential to $ \mathbb{S}^2 $, see \cite{Alves2002on} for instance. In this paper, we define the far-field pattern $ u^{\infty}(\hat{x}) $of the scattered field $u$, that is
\begin{equation}
\nonumber
u^{\infty}(\hat{x}):=u_p^{\infty}(\hat{x})+u_s^{\infty}(\hat{x}).
\end{equation}

In the case $\mu_i=\mu_e=\mu>0$ and $\lambda_i=\lambda_e=\lambda$ satisfies $3\lambda+2\mu>0$, then the problem (\ref{eq1})-(\ref{rd}) becomes the following scattering system
\begin{alignat}{4}
\label{m1} \Delta^* v(x)+ \rho(x)\omega^2 v(x)=&0,\  &x&\in \mathbb{R}^3 \setminus \overline{D_b},\\
\label{m2} v(x)=v^{sc}(x)+u^{inc}(x)&,\              &x&\in \mathbb{R}^3 \setminus \overline{D}, \\      
\label{mbd} v(x)=&0, \ &x&\in\partial D_b,\\
\label{mrd} \lim_{r=|x|\rightarrow \infty }r\left[\frac{\partial v^{sc}_\alpha(x)}{\partial r}-i k_e^\alpha v^{sc}_\alpha(x)\right]=&0,\ &\alpha &=p,s.
\end{alignat}
where $\Delta^*=\mu\Delta +(\lambda+\mu)\nabla(\nabla\cdot)$ and $v^{sc}(x)=  v_p(x)+ v_s(x)$.

In condition (\ref{rd}) and (\ref{mrd}), by Helmholtz decomposition theorem, we denoted by $v_p$ and $v_s$
 to be the longitudinal and transversal parts of the scattered field $v^{sc}$ 
 \begin{equation}
 \nonumber
 v_p=-(k_e^p)^{-2}\nabla (\nabla\cdot v^{sc}),\ \ v_s=-(k_e^s)^{-2}\nabla\times(\nabla\times v^{sc}).
 \end{equation}

The free space fundamental solution of the Lam\'{e} system in $\mathbb{R}^3$ is given by
\begin{equation}\label{fundamental}
  \Gamma_e(x,y):=\frac{(k_e^s)^{2}}{4\pi\omega ^2}\frac{e^{ik_e^s|x-y|}}{|x-y|}\emph{I}
  +\frac{1}{4\pi\omega^2}\nabla_x\nabla_x^{\mathsf T}\left[\frac{e^{ik_e^s|x-y|}}{|x-y|}-\frac{e^{ik_e^p|x-y|}}{|x-y|}\right]
\end{equation}
where $'\mathsf T'$ denotes transposition, \emph{I} is the identity matrix and
$\Phi_p(x,y)=\frac{e^{ik_e^p|x-y|}}{4\pi|x-y|}$ is fundamental solution to the Helmholtz equation.
\section{The direct scattering problem}\label{sec2}
\setcounter{equation}{0}
The direct scattering problem is to find the field $(v,u)$(or $v$) satisfying the conditions (\ref{bd1})-(\ref{rd})(or (\ref{mbd})-(\ref{mrd})) with given scatterer $(D, \rho, D_b)$ and incident plane wave $u^{inc}$. In this section, we deal with the well-posedness of the direct transmission problem (\ref{eq1})-(\ref{rd}) and medium scattering problem (\ref{m1})-(\ref{mrd}). Also, we establish a priori estimate of solution for both problems with boundary data in $L^p (1< p \leq 2)$ which are necessary later in the uniqueness proof of the Inverse Problem .
\begin{theorem}
   The transmission problem (\ref{eq1})-(\ref{rd}) has at most one solution.
\end{theorem}
\begin{proof}
Let $u^{inc}(x)=0$, which means $f(x)=h(x)=0$. Suppose $(v,u)$ is a solution of the transmission problem
(\ref{eq1})-(\ref{rd}), then from radiation condition (\ref{rd}), we have the following results \cite[p.127]{Kupradze}
\begin{align}
 \label{asym1} u_p(x)= \mathit{O}(|x|^{-1}), \qquad u_s(x)= \mathit{O}(|x|^{-1}), \\
  T_e u_p(x)-ik_p(\lambda_e+2\mu_e)u_p(x)= \mathit{O}(|x|^{-2}), \\
  T_e u_s(x)-ik_s\mu_eu_s(x)= \mathit{O}(|x|^{-2}), \\
 \label{asym2} \overline{u_p(x)}u_s(x)= \mathit{O}(|x|^{-3}), \qquad u_p(x)\overline{u_s(x)}= \mathit{O}(|x|^{-3}).
\end{align}
when $|x|\rightarrow\infty$.

Set $B_R:=\{x\in\mathbb{R}^3: |x|<R\}$ is a ball containing $D$, $S_R :=\{x\in\mathbb{R}^3: |x|=R\}$ and $D_R:=B_R\setminus D$.
Then we apply Gauss' theorem in $D_R$ to get
\begin{equation}\label{int1}
\begin{split}
&\int_{D_R}\left( \overline{u(x)}\Delta_e^*u(x)-u(x)\Delta_e^*\overline{u(x)}\right) {\rm d}  x\\
  =&\int_{\partial D}\left( \overline{u(x)}T_e u(x)-u(x)T_e\overline{u(x)}\right) {\rm d} s(x)-\int_{S_R}\left( \overline{u(x)}T_e u(x)-u(x)T_e\overline{u(x)}\right) {\rm d} s(x).
\end{split}
\end{equation}
Here, $\Delta_e^*=\mu_e\Delta +(\lambda_e+\mu_e)\nabla(\nabla\cdot)$. Set
\begin{equation}
\nonumber
 A u_p(x):=T_e u_p(x)-ik_p(\lambda_e+2\mu_e)u_p(x),\quad B u_s(x):=T_e u_s(x)-ik_s\mu_eu_s(x).
\end{equation} 
Then direct calculation can obtain
\begin{align*}
 & \int_{\partial D}\left(\overline{u(x)}T_e u(x)-u(x)T_e\overline{u(x)}\right) {\rm d} s(x)=\int_{S_R}\left( \overline{u(x)}T_e u(x)-u(x)T_e\overline{u(x)}\right) {\rm d} s(x) \\
  = & \int_{S_R} \left\lbrace \overline{[u_p(x)+u_s(x)]}A u_p(x)+\overline{[u_p(x)+u_s(x)]}B u_s(x)\right\rbrace  {\rm d} s(x)\\
 & -\int_{S_R}\left\lbrace[u_p(x)+u_s(x)]\overline{A u_p(x)}+[u_p(x)+u_s(x)]\overline{B u_s(x)}\right\rbrace  {\rm d} s(x)\\
 & + 2ik_p(\lambda_e+2\mu_e)\int_{S_R}|u_p(x|^2{\rm d} s(x) + 2ik_s\mu_e \int_{S_R}|u_s(x|^2{\rm d} s(x)\\
  &+2i[k_s\mu_e+k_p(\lambda_e+2\mu_e)]\int_{S_R}\Re(u_p(x)\overline{u_s(x)}){\rm d}s(x)
\end{align*}
From interfacial conditions (\ref{bd1}) and boundary condition (\ref{bd3}), the left side of the above equation is
\begin{align*}
 & \int_{\partial D}\left( \overline{u(x)}T_e u(x)-u(x)T_e\overline{u(x)}\right) {\rm d}s(x)\\
= & \int_{\partial D}\left( \overline{u(x)}T_i u(x)-u(x)T_i\overline{u(x)}\right) {\rm d}s(x)+\int_{\partial D_b}\left( \overline{u(x)}T_i u(x)-u(x)T_i\overline{u(x)}\right) {\rm d}s(x)\\
= & \int_{\partial D_i}\left\lbrace \overline{u(x)}\left[ \Delta_i^*u(x)+ \rho \omega^2u(x)\right] -u(x)\left[ \Delta_i^*\overline{u(x)}+\rho \omega^2\overline{u(x)}\right] \right\rbrace {\rm d}s(x)\\
=&0.
\end{align*}
Together with (\ref{asym1})-(\ref{asym2}), it's easy to obtain
\begin{equation}\label{rel}
 \lim_{R\rightarrow \infty}\int_{S_R}|u_p(x)|^2{\rm d}s(x)=0,\quad  \lim_{R\rightarrow \infty}\int_{S_R}|u_s(x)|^2{\rm d}s(x)=0.
\end{equation}
Since $u_p(x), u_s(x)$ satisfy vector Helmholtz equation,   Relic lemma \cite{Colton2013} tells us that $u_p(x)= u_s(x)=0$ in $D_e$.
Then it is obvious that $u(x)=T_e u(x)=0$ on $\partial D$. According to interfacial conditions (\ref{bd1}), we can get $v(x)=T_i v(x)=0$ on $\partial D$.

In order to show that $v(x)=0$ in $D_i$, we employ a unique continuation principle. We choose a continuously differentiable
 extension of $\rho (x)$ into $\overline{B_R}$ and define $v(x)=0$ for $x\in B_R\setminus D$. Then
$v\in H^1(B_R\setminus \overline{D_b})$ is a weak solution to $\Delta_i^* v+\rho(x)\omega^2 v=0 $ in $B_R\setminus \overline{D_b}$
and therefore it is in $H^2(B_R\setminus \overline{D_b})$ by the usual regularity theorems. Now we apply the unique
continuation principle from \cite{Hormander1985The} to get $v(x)=0$ in $D_i$.
\end{proof}

To show existence we formulate the transmission problem (\ref{eq1})-(\ref{rd}) into systems of volume  and surface integral equations. Here, we recall the definitions of elastic single- and double-layer potentials
\begin{align*}
  (S_e\varphi)(x) & := \int_{\partial D}\Gamma_e(x,y)\varphi(y){\rm d}s(y) , &&  x\in D_e\setminus\partial D,\\
    (V_e\varphi)(x) & := \int_{\partial D}T_{e,y}[\Gamma_e(x,y)]^{\mathsf T}\varphi(y){\rm d}s(y) , && x\in D_e\setminus\partial D,
\end{align*}
where $\Gamma_e(x,y)$ is the fundamental solution of $\Delta_e^* u+\omega^2 u=0 $ and $T_{e,y}$ means $T_e$ applied at $y\in \partial D$. It is well known that $S_e$ and $T_{e,x}V_e$ are continuous in
$\mathbb{R}^3$ but $V_e$ and $T_{e,x}S_e$ have jumps on the boundary $\partial D$(details can be found in \cite{Kupradze}).
\begin{align*}
(S_e\phi)(x)& = \int_{\partial D}\Gamma_e(x,y)\varphi(y){\rm d}s(y) , &&  x\in \partial D,\\
T_{e,x}(S_e\phi)(x) &=\int_{\partial D}T_{e,x}[\Gamma_e(x,y)]^{\mathsf T}\varphi(y){\rm d}s(y))-\frac{1}{2}\varphi(x) , &&  x\in \partial D,\\
(V_e\phi)(x)& = \int_{\partial D}T_{e,y}[\Gamma_e(x,y)]^{\mathsf T}\varphi(y){\rm d}s(y)+\frac{1}{2}\varphi(x) , && x\in\partial D,\\
T_{e,x}(V_e\phi)(x) & =T_{e,x} \int_{\partial D}T_{e,y}[\Gamma_e(x,y)]^{\mathsf T}\varphi(y){\rm d}s(y), && x\in\partial D.
\end{align*}
Then, we introduce the integral operators
\begin{align*}
  (S_{ee}\phi)(x) & := \int_{\partial D}\Gamma_e(x,y)\varphi(y){\rm d}s(y) , &&  x\in \partial D,\\
    (K_{ee}\phi)(x) & := \int_{\partial D}T_{e,y}[\Gamma_e(x,y)]^{\mathsf T}\varphi(y){\rm d}s(y) , && x\in \partial D,\\
  (K_{ee}'\phi)(x) & := \int_{\partial D}T_{e,x}[\Gamma_e(x,y)]^{\mathsf T}\varphi(y){\rm d}s(y)) , && x\in \partial D,\\
   (N_{ee}\phi)(x) & :=T_{e,x} \int_{\partial D}T_{e,y}[\Gamma_e(x,y)]^{\mathsf T}\varphi(y){\rm d}s(y) , && x\in \partial D.
\end{align*}
All the above integrals are well defined in the sense of the principal value. And in particular the operator $S_{ee}$ for $x\in \partial D$ has a singularity of form $|x-y|^{-1}$, the operators ${K_{ee}}$, ${K_{ee}'}$ have singularity of form $|x-y|^{-2}$ and $N_{ee}$ admits a hyper singular kernel of form $|x-y|^{-3}$. Moreover, we introduce an antisymmetric matrix operator 
\begin{align*}
 (W^{\tau}_{ee}\phi)(x) &:= \int_{\partial D}T^{\tau}_{e,y}[\Gamma_e(x,y)]^{\mathsf T}\varphi(y){\rm d}s(y),
   && x\in \partial D,  \\
 (W^{'\tau}_{ee}\phi)(x) &:= T^{\tau}_{e,x}\int_{\partial D}[\Gamma_e(x,y)]^{\mathsf T}\varphi(y){\rm d}s(y),
   && x\in \partial D,  \\
 (V^{\gamma,\tau}_{ee}\phi)(x) &:= T^{\gamma}_{e,x}\int_{\partial D}T^{\tau}_{e,y}[\Gamma_e(x,y)]^{\mathsf T}\varphi(y){\rm d}s(y) , && x\in \partial D
\end{align*}
where $\tau,\gamma$ are arbitrary constants, and $T^{\tau}_e$ is a generalised traction operator\cite[H.\uppercase\expandafter{\romannumeral 1}, 13.5]{Kupradze} defined by
\begin{align*}
T_\alpha^{\tau} & := ||T_{\alpha,lm}^{\tau}||_{3\times3}=(\mu_\alpha+\tau) \textbf{n} \cdot \text{grad}+(\lambda_\alpha+\mu_\alpha-\tau )\textbf{n}\text{div} + \tau \textbf{n} \times \text{curl}, \\
 T_{\alpha,lm}^{\tau} & =\delta_{lm}\mu_\alpha\frac{\partial}{\partial n}+(\lambda_\alpha+\mu_\alpha-\tau)n_l\frac{\partial}{\partial x_m}+\tau n_m \frac{\partial}{\partial x_l},\ \alpha=i,e, \ l,m=1,2,3.
\end{align*}
Obviously, $T_\alpha^{\mu_\alpha}=T_\alpha$. $W^{\tau}_{ee}$ has the same jump and regularity properties 
as double-layer operator ${K_{ee}}$. Similarly, we also introduce the operators $S_{ii},K_{ii},K_{ii}'$ and $N_{ii}, W^{\tau}_{ii}$ defined on $\partial D_b$, it follows from Theorem 3.36 in \cite[H.\uppercase\expandafter{\romannumeral 4}]{Kupradze} that the operators $S_{\alpha\alpha},K_{\alpha\alpha},K'_{\alpha\alpha}$ with $\alpha=i,e$ are both bounded in $L^p(\partial D)(1<p<\infty)$.

Moreover, we define the operators $S_{ei},K_{ei},K_{ei}'$ and $N_{ei}, W^{\tau}_{ei}$ on $\partial D$ in the same way, but with $x\in \partial D_b$. The operators $S_{ie},K_{ie},K_{ie}'$ and $N_{ie}, W^{\tau}_{ie}$ are defined on $\partial D_b$, but $x\in \partial D$. All above operators are well-defined with no singularity.

Then we give the following existence result.  
\begin{theorem}\label{existence}
 For $f,h\in L^p(\partial D)$ with $4/3\leq p\leq 2$, the transmission problem (\ref{eq1})-(\ref{rd}) has a unique solution
   $(v,u)\in L^2(D\setminus \overline{D_b})\times L^2(\mathbb{R}^3 \setminus D)$ satisfying
  \begin{equation}\label{es}
     ||v||_{L^2 (D\setminus \overline{D_b})} +||u||_{L^2 (\mathbb{R}^3 \setminus \overline{D})} \leq C(||f||_{L^p (\partial D) } +||h||_{L^p{(\partial D)}})
  \end{equation}
 \end{theorem}
 \begin{proof}
Step 1. Assume that $\rho(x)\omega^2\equiv\omega_1^2 $ is a constant. We seek a pair of solution for the
problem (\ref{eq1})-(\ref{rd}) in the form
\begin{equation}\label{eq2.1}
\begin{split}
    v(x)=&\int_{\partial D_b}T^{\kappa_i}_{i,y}[\Gamma_i(x,y)]^{\mathsf T}\eta(y){\rm d}s(y)+ \int_{\partial D}\Gamma_i(x,y)\alpha_i\psi(y){\rm d}s(y)\\
      &\ +\int_{\partial D}T^{\kappa_i}_{i,y}[\Gamma_i(x,y)]^{\mathsf T}\beta_i\varphi(y){\rm d}s(y),\qquad x\in  D_i\\
    u(x)=&\int_{\partial D}\Gamma_e(x,y)\alpha_e\psi(y){\rm d}s(y)+\int_{\partial D}T^{\kappa_e}_{e,y}[\Gamma_e(x,y)]^{\mathsf T}\beta_e\varphi(y){\rm d}s(y),\qquad x\in  D_e
\end{split}
\end{equation}
where $\alpha_r,\beta_r$ and $\kappa_r(r=i,e)$ are arbitrary constants to be determined below, $\varphi(y),
\psi(y)$ and $\eta(y)$ are the unknown vectors.

Taking into account the identities (see\cite[H.\uppercase\expandafter{\romannumeral 5}, 1.6]{Kupradze})
\begin{equation}\label{id1}
 \begin{split}
    T^{\gamma_e}_eu(x) &=T_eu(x)+(\gamma_e-\mu_e)\mathcal{M}u(x), \\
    T^{\gamma_i}_iv(x) &=T_iv(x)+(\gamma_i-\mu_i)\mathcal{M}v(x),
 \end{split}
\end{equation}
where $\gamma_e $ and $\gamma_i$ are arbitrary constants, and matrix $\mathcal{M}$ is defined as
\begin{equation}
\nonumber
\begin{split}
\mathcal{M} &:=||\mathcal{M}_{lm}||_{3\times3}, \\
\mathcal{M}_{lm}& =n_m \frac{\partial}{\partial x_l}- n_l\frac{\partial}{\partial x_m},\ l,m=1,2,3.
\end{split}
\end{equation}
then we have
\begin{equation}\label{bd2.1}
T^{\gamma_i}_iv(x)-T^{\gamma_e}_eu(x)=T_iv(x)-T_eu(x)+ \mathcal{M}[(\gamma_i-\mu_i)v(x)-(\gamma_e-\mu_e)u(x)]
\end{equation}
Let
\begin{equation}
\nonumber
\gamma_i-\mu_i=\gamma_e-\mu_e
\end{equation} 
then we get a equivalent form of interfacial conditions (\ref{bd1}) 
\begin{equation}\label{bd2.2}
  \begin{split}
     v(x)-u(x) &=F(x), \\
       T^{\gamma_i}_iv(x)-T^{\gamma_e}_eu(x)&=H(x),
  \end{split}
\end{equation}
where
\begin{equation}\label{bd2.3}
\begin{split}
   F(x)&=f(x), \\
   H(x)&=h(x)+(\gamma_i-\mu_i)\mathcal{M}F(x).
\end{split}
\end{equation}
We choose $\alpha_i, \beta_i, \kappa_i$ and $\gamma_i$ as (see \cite[H. \uppercase\expandafter{\romannumeral12}, \S 2]{Kupradze})
\begin{gather}\label{ab}
   \alpha_r=\frac{\mu_i\mu_e}{\mu_r(\mu_i+\mu_e)} ,\qquad \beta_r=-\alpha_r, \qquad b_r=\frac{\lambda_r+\mu_r}{\lambda_r+2\mu_r}, \qquad r=i,e \\
    \kappa_e =\frac{\mu_e(b_i\mu_i-b_e\mu_e)}{b_i\mu_i+b_e\mu_e},\qquad  \kappa_i =\frac{\mu_i(b_e\mu_e -b_i\mu_i)}{b_i\mu_i+b_e\mu_e} \\
     \gamma_e=\frac{\mu_e(\mu_e-\mu_i)}{\mu_i+\mu_e},\qquad \gamma_i=\frac{\mu_i(\mu_i-\mu_e)}{\mu_i+\mu_e}
\end{gather}
Then, by the jump relations of the single- and double-layer potentials (see \cite{Kupradze}) and above
consequences, the problem (\ref{eq1})-(\ref{rd}) can be reduced to the system of integral equations
\begin{equation}\label{ind}
  \begin{pmatrix}
  \eta\\
  \varphi\\
  \psi
  \end{pmatrix}+ L\begin{pmatrix}
  \eta\\
  \varphi\\
  \psi
  \end{pmatrix}=\begin{pmatrix}
  0\\
  -2F\\
  -2H
  \end{pmatrix}\quad \mathrm{in}\ L^p(\partial D_b)\times L^p(\partial D)\times L^p(\partial D),
\end{equation}
where the operator $L$ is defined as 
\begin{equation}\label{ind}
 L=\begin{pmatrix}
 -2W_{ii}^{\kappa_i}&2\alpha_iW_{ei}^{\kappa_i,i}&\-2 \alpha_iS^i_{ei}\\
 -2W_{ie}^{\kappa_i}&2(\alpha_iW_{ee}^{\kappa_i,i}-\alpha_eW_{ee}^{\kappa_e})&2(\alpha_eS_{ee}-\alpha_iS^i_{ee})\\
 -2V^{\gamma_i,\kappa_i,i}_{ie}&2(\alpha_iV^{\gamma_i,\kappa_i,i}_{ee}-\alpha_eV^{\gamma_e,\kappa_e}_{ee})&2(\alpha_eW_{ee}^{'\gamma_e}-\alpha_iW_{ee}^{'\gamma_i,i})
 \end{pmatrix}.
\end{equation}

Here, the operator ${W}_{th}^{\kappa_i,i}, S^i_{th}, W'^{,i}_{th},$ and $ V^{\gamma_i,\kappa_i,i}_{ee}$ with $t,h=i,e$ are defined similarly as $W_{th}^{k_i}, S_{th}, W'_{th},$ and $ V^{\gamma_i,\kappa_i}_{ee} $ with the kernel $\Gamma_e(x,y)$ replaced by $\Gamma_i(x,y)$. Using the same procedure as in \cite{Kupradze}, we may calculate the symbolic determinant of (\ref{ind})(see details in Appendix B); we know that the system (\ref{ind}) is normally solvable and Fredholm's theorems hold for it. This together with the uniqueness of the scattering problem (\ref{eq1})-(\ref{rd}) implies that (\ref{ind}) has a unique solution $(\eta,\psi,\phi)\in L^p(\partial D_b)\times L^p(\partial D)\times L^p(\partial D)$ satisfying the estimate
\begin{equation}\label{ind.1}
  ||\eta||_{L^p{(\partial D_b)}}+||\varphi||_{L^p{(\partial D)}}+||\psi||_{L^p{(\partial D)}}\leq C( ||F||_{L^p{(\partial D)}} +||H||_{L^p{(\partial D)}})
\end{equation}
Then, by the fact that $\Gamma_i(x,y)$ has a singularity of form $|x-y|^{-1}$ and volume potential operator is bounded from $L^2(D)$ into $W^{2,2}(D)$ (see\cite[H.\uppercase\expandafter{\romannumeral 4},Theorem 2.3]{Kupradze}, \cite[Theorem 5.6]{peter1998On}), we have
\begin{equation}\label{Les1}
\begin{split}
&\left|\left|\int_{\partial D}\Gamma_i(x,y)\psi(y){\rm d}s(y)\right|\right|_{L^2(D)}\\
=& \mathop{\mathrm{sup}} \limits_{g\in L^2(D),||g||_{L^2(D)}=1}\left|\int_{D}\left\lbrace\int_{\partial D}\Gamma_i(x,y)\psi(y){\rm d}s(y)\right\rbrace g(x){\rm d}x\right|\\
=&\mathop{\mathrm{sup}} \limits_{g\in L^2(D),||g||_{L^2(D)}=1}\left|\int_{\partial D}\left\lbrace\int_{ D}\Gamma_i(x,y)g(x){\rm d}x\right\rbrace \psi(y){\rm d}s(y)\right|\\
\leq & C \mathop{\mathrm{sup}} \limits_{y\in \partial D}||\Gamma_i(x,y)||_{L^2(D)}||\psi||_{L^p(\partial D)}\\
\leq & C ||\psi||_{L^p(\partial D)}
\end{split}
\end{equation}
where $C$ is a constant depending on $D$. The same procedure, with the fact that the boundary trace operator is bounded from $W^{1,2}(D)$ into $L^q(D)$ for $2\leq q\leq 4$ (see\cite[Theorem 5.36]{Adams2003Sobolev}), we derive 
\begin{equation}\label{Les2}
\begin{split}
&\left|\left|\int_{\partial D}T^{\kappa_i}_{i,y}[\Gamma_i(x,y)]^{\mathsf T}\varphi(y){\rm d}s(y)\right|\right|_{L^2(D)}\\
=& \mathop{\mathrm{sup}} \limits_{g\in L^2(D),||g||_{L^2(D)}=1}\left|\int_{D}\left\lbrace \int_{\partial D}T^{\kappa_i}_{i,y}[\Gamma_i(x,y)]^{\mathsf T}\varphi(y){\rm d}s(y)\right\rbrace g(x){\rm d}x\right|\\
=&\mathop{\mathrm{sup}} \limits_{g\in L^2(D),||g||_{L^2(D)}=1}\left|\int_{\partial D}\left\lbrace \int_{ D}T^{\kappa_i}_{i,y}[\Gamma_i(x,y)]^{\mathsf T}g(x){\rm d}x\right\rbrace  \varphi(y){\rm d}s(y)\right|\\
\leq &  \mathop{\mathrm{sup}} \limits_{g\in L^2(D),||g||_{L^2(D)}=1}\left|\left|T^{\kappa_i}_{i,y}\int_{\partial D}[\Gamma_i(x,y)]^{\mathsf T}g(x){\rm d}x\right|\right|_{L^q(\partial D)}||\varphi||_{L^p(\partial D)}\\
\leq & C ||\varphi||_{L^p(\partial D)}
\end{split}
\end{equation}
where $C$ is a constant depending on $D.$ Further, we have
\begin{equation}\label{Les3}
\begin{split}
&\left|\left|\int_{\partial D_b}T^{\kappa_i}_{i,y}[\Gamma_i(x,y)]^{\mathsf T}\eta(y){\rm d}s(y)\right|\right|_{L^2(D)}\leq C ||\eta||_{L^p(\partial D_b)}\\
&\left|\left|\int_{\partial D}\Gamma_e(x,y)\psi(y){\rm d}s(y)\right|\right|_{L^2(D)}\leq C ||\psi||_{L^p(\partial D)}\\ 
&\left|\left|\int_{\partial D}T^{\kappa_e}_{e,y}[\Gamma_e(x,y)]^{\mathsf T}\varphi(y){\rm d}s(y)\right|\right|_{L^2(D)}\leq C ||\varphi||_{L^p(\partial D)}
\end{split}
\end{equation}
Then combine (\ref{eq2.1}) and (\ref{ind.1})-(\ref{Les3}), we obtain estimate (\ref{es}) in the case when $\rho(x)\omega^2\equiv\omega_1^2 $.

Step 2. For the general case $\rho(x)\in L^\infty {(D_i)}$, we consider the following problem
\begin{align}
\label{eq11} \mu_i\Delta V(x)+(\lambda_i+\mu_i)\nabla(\nabla\cdot V(x))+ \rho(x)\omega^2V(x)  =&g, \qquad \quad x\in  D_i,\\
 \mu_e\Delta U(x)+(\lambda_e+\mu_e)\nabla(\nabla\cdot U(x))+ \omega^2 U(x)  =&0, \qquad \quad  x\in  D_e,\\
\label{bd11} V(x)-U(x)  = &0, \qquad \quad  x\in \partial D , \\
\label{bd22} T_i V(x)-T_e U(x) =&0,\qquad \quad  x\in \partial D , \\
 \label{bd33} V(x) =  &0,\qquad \quad x\in \partial D_b,\\
\label{rdd} \lim_{r\rightarrow \infty }r(\frac{\partial U_\alpha}{\partial r}-i k_\alpha^e U_\alpha) =&0,\qquad\quad \alpha= p,s,
\end{align}
where $g:= (\omega_1^2-\rho(x)\omega^2) \widetilde{v} \in L^2 {( D_i)} $, and $( \widetilde{v}, \widetilde{u})$
 denotes the solution of the problem (\ref{eq1})-(\ref{rd}) with $\rho(x)\omega^2\equiv\omega_1^2 $. By step 1, we have
\begin{equation}\label{es1}
 ||\widetilde{v}||_{L^2{(D\setminus \overline{D_b})}}+|| \widetilde{u}||_{L^2{(\mathbb{R}^3 \setminus \overline{D})}}\leq C(||f||_{L^p{(\partial D)}} +||h||_{L^p{(\partial D)}})
\end{equation}
By using the variational method  \cite{Charalambopoulos2002On}, it is easy to prove that for every $g\in L^2_{( D_i)}$
the problem (\ref{eq11})-(\ref{rdd}) has a unique solution $(V,U)\in H^1_{(D\setminus \overline{D_b})}\times H^1_{(\mathbb{R}^3 \setminus D)} $
 satisfying the estimate
\begin{equation}\label{es2}
    ||V||_{H^1 {(D\setminus \overline{D_b})}}+||U||_{H^1{(\mathbb{R}^3 \setminus\overline{D})}}\leq C||g||_{L^2{(D\setminus \overline{D_b})}}.
\end{equation}
Define $v:=V+ \widetilde{v}$ and $u:=U+ \widetilde{u}$. Then from (\ref{es1}) and (\ref{es2}), we can easily get
that $(v,u)\in L^2 {(D\setminus \overline{D_b})}\times L^2 {(\mathbb{R}^3 \setminus\overline{D})}$ is the unique solution of the problem
 (\ref{eq1})-(\ref{rd}) satisfying the estimate (\ref{es}). The proof is thus complete.
 \end{proof}

Next, we give estimates of solutions to transmission problem (\ref{eq1})-(\ref{rd}) with  incident point sources.
\begin{corollary}\label{coro}
  For $z^*\in \partial D$ and for a sufficiently small constant $\delta>0$ , we define\\
  $ z_j:=z^*+ \frac{\delta}{j} n(z^*)\in \mathbb{R}^3\setminus\overline{D},\ (j=1,2,...)$. Let $(v_j,u_j)$
   be the solution of the transmission problem (\ref{eq1})-(\ref{rd}) corresponding to the incident point
   source $u_j^{inc}=\frac{\Gamma^p(x,z_j)\cdot q}{||\nabla_x\nabla_x^{\mathsf{T}}\Phi_p(x,z_j)\cdot q||_{L^2(\partial D)}}$
    where $\Gamma^p(x,z_j) :=\nabla_x\nabla_x^{\mathsf{T}}\cdot\Gamma_e(x,z_j) =\frac{(k_e^p)^2}{\omega^2}\nabla_x\nabla_x^{\mathsf{T}}\Phi_p(x,z_j), \ j\in \mathbb{N}$ and $q=n(z^*)$. Then
  \begin{equation}\label{coroes}
    ||v_j||_{L^2(D\setminus \overline{D_b})}\leq C
  \end{equation}
  uniformly for $j\in \mathbb{N}$.
\end{corollary}

\begin{proof}
  Let $U_j:=u_j- \frac{\nabla_x\nabla_x^{\mathsf{T}}\Gamma_e(x,y_j)\cdot q}{||\nabla_x\nabla_x^{\mathsf{T}}\Gamma_e(x,z_j)\cdot q||_{L^2(\partial D)}}$ and
  $V_j:=v_j$ with $y_j:=z^*- \frac{\delta}{j} n(z^*)\in D\setminus \overline{D_b}$. Then $(V_j,U_j)$ satisfies
  \begin{align*}
 \Delta_i V_j(x)+ \rho(x)\omega^2V_j(x)=&0,   &x&\in D_i,\\
 \Delta U_j(x)+ \omega^2 U_j(x)=&0,     &x&\in D_e,\\
  V_j(x)-U_j(x)= &f_j(x),  &x&\in \partial D , \\
  T_i V_j(x)-T_e U_j(x))=&h_j(x),  &x&\in \partial D , \\
  V_j(x)= &0,  &x&\in\partial D_b,\\
  \lim_{r\rightarrow \infty }r(\frac{\partial U_{j,\alpha}}{\partial r}-i k_\alpha^e U_{j,\alpha})=&0,  &\alpha &=p,s,
  \end{align*}
with
\begin{align}\label{fh}
  f_j(x) & =\frac{1}{||\nabla_x\nabla_x^{\mathsf{T}}\Phi_p(x,z_j)\cdot q||_{L^2(\partial D)}}\nabla_x\nabla_x^{\mathsf{T}}[ \Phi_p(x,z_j) +\Phi_p(x,y_j)]\cdot q ,\\
  h_j(x) & =\frac{1}{||\nabla_x\nabla_x^{\mathsf{T}}\Phi_p(x,z_j)\cdot q||_{L^2(\partial D)}}T_e\nabla_x\nabla_x^{\mathsf{T}}[ \Phi_p(x,z_j) +\Phi_p(x,y_j)]\cdot q.
\end{align}
To estimate the boundedness of $f, h$, we first calculate two derivatives of $\Phi_p(x,z_j)$ and set $x=(x_1,x_2,x_3)$
\begin{equation}\label{derivative}
\frac{\partial^2\Phi_p(x,y)}{\partial x_l\partial x_m} =\frac{e^{ik_p|x-y|}}{4\pi |x-y|} \left\lbrace \frac{(ik_p-1)\delta_{lm}}{|x-y|^2}+\frac{(x_l-y_l)(x_m-y_m)}{|x-y|^2}[\frac{3(1-ik_p|x-y|)}{|x-y|^2}-k_p^2]\right\rbrace ,\ l,m=1,2,3
\end{equation}
Applying the mean-value theorem, we obtain the estimate
\begin{equation}
\left|(1-ik_p|x-y|)e^{ik_p|x-y|}-1\right|\leq {k_p}^2|x-y|^2
\end{equation}
Then, we can easily get that there exists a constant $C(\lambda_e,\mu_e)$ such that
\begin{equation}\label{singular}
   \left|\frac{\partial^2\Phi_p(x,y)}{\partial x_l\partial x_m}\right|\leq C(\lambda_e,\mu_e)|x-y|^{-1}
\end{equation}
Then it's clear that $f_j\in L^2(\partial D)$ is uniformly bounded for $j\in \mathbb{N}.$ Further, $\Gamma^p$ is part of the fundamental solution, by direct calculation  \cite{Athanasiadis20083D} we have the following estimate
\begin{equation}\label{singular1}
 T_e\Gamma^p(x,y)-ik_p(\lambda+2\mu)\Gamma^p(x,y)= \mathit{O}\left( |x-y|^{-2}\right) 
\end{equation}
Together with (\ref{singular}), we obtain that $h_j$ is uniformly bounded in $L^2(\partial D)$ for all $j\in \mathbb{N}.$
Then we get estimate (\ref{coroes}) from Theorem \ref{existence}.
\end{proof}

For scattering problem (\ref{m1})-(\ref{mrd}), the well-posedness was obtained for a more general case in \cite[Proposition 2.1]{Bai2021Effective}.
\begin{proposition}
There exists a unique solution $u \in H^1(\mathbb{R}^3 \setminus \overline{D})$ to the scattering problem $(\ref{m1})-(\ref{mrd})$. Furthermore, it holds that
\begin{equation}
\nonumber
||v||_{H^1(\mathbb{R}^3 \setminus \overline{D_b})}\leq C (||u^{inc}||_{H^{1/2}(\partial D)}+||T_e u^{inc}||_{H^{-1/2}(\partial D)})
\end{equation}
where $C$ is a positive constant.
\end{proposition}

Next, we will give a $L^2$ estimate for problem (\ref{m1})-(\ref{mrd}).
\begin{theorem}\label{medium}
Let $v$ be the solution of the scattering problem (\ref{m1})-(\ref{mrd}) corresponding the point source $u^{inc}=\nabla\Phi_p(x,y)$, where $z\in \mathbb{R}^3 \setminus \overline{D}$ and $\Phi_p(x,y)=e^{ik_p|x-y|}/(4\pi |x-y|)$. Then $v\in  L^2(D\setminus \overline{D_b})$ and $v-u^{inc}\in H^1(D\setminus \overline{D_b})$ such that
\begin{equation}
||v(x;z)||_{L^2(D\setminus \overline{D_b})}+||v(x;z)-u^{inc}||_{H^1(D\setminus \overline{D_b})}
\leq C(||u^{inc}(x;z)||_{ L^2(D\setminus \overline{D_b})}+||u^{inc}(x;z)||_{L^{\infty}(\partial {D_b})}) 
\end{equation}
where $C$ is a constant.
\end{theorem}

To prove the estimate, we reformulate the scattering problem (\ref{m1})-(\ref{mrd}) as an equivalent Lippmann-Schwinger type equation. In the case $\rho(x)=1$, problem (\ref{m1})-(\ref{mrd}) changes into
\begin{equation}\label{bdeq}
\begin{cases}
\Delta^* w(x)+ \omega^2 w(x)=0,\  &x \in \mathbb{R}^3 \setminus \overline{D_b},\\
 w(x)= f, \ &x \in\partial D_b,\\
 \lim_{r=|x|\rightarrow \infty }r[\frac{\partial w_\alpha(x)}{\partial r}-i k_e^\alpha w_\alpha(x)]=0,\ &\alpha =p,s.
\end{cases}
\end{equation}
where $w(x)=w_p(x)+w_s(x)$, $ w_p(x)$ and $w_s(x) $ are longitudinal and transversal part of $w(x)$ respectively . It was proved in \cite{peter1993uniqueness} that the problem (\ref{bdeq}) is well-posed for every $f$ belonging to a suitable function space.
  
  Suppose $G(x,y)$ is the Green function of problem (\ref{bdeq}) in $\mathbb{R}^3 \setminus \overline{D_b}$ with 
  $f=0$ and $V(x,y)$ is the solution to the problem (\ref{bdeq}) with $f= \nabla\Phi_p(x,y)$ with  $y\in \mathbb{R}^3 \setminus \overline{D_b}$. Define 
 \begin{equation}
 \nonumber
 V^{inc}:=  \nabla \Phi_p(x,y) + V(x;y)\ \ \texttt{in}\ \mathbb{R}^3 \setminus \overline{D_b}.
 \end{equation}
Then applying representation formula \cite[Theorem 5.3]{peter1998On}, we have
\begin{equation} 
\begin{split}
v(x;z)=&\left( \int_{\partial D}-\int_{\partial D_b}\right) \left\lbrace  G(x,y)T_e v(y;z)-v(y;z)T_{e,y} G(x,y)\right\rbrace  \mathrm{d} s(y)\\
     & -\int_{D\setminus \overline{D_b}} \omega^2 \left[ 1-\rho(y)\right]  G(x,y)v(y;z)\mathrm{d} y,\ \ x\in D\setminus \overline{D_b}
\end{split}
\end{equation}
Since $G(x,y)=v(y;z)=0$ for $y \in \partial D_b$ and the fact that $v^{sc}(\cdot;z), G(x,\cdot)$ are solutions satisfying Kuparadze radiation condition, we have
\begin{equation}
\nonumber
\int_{\partial D}\left\lbrace  G(x,y)T_e v^{sc}(y;z)-v^{sc}(y;z)T_{e,y} G(x,y)\right\rbrace  \mathrm{d} s(y)=0.
\end{equation}
Then combine above equations, the solution $v$ of the scattering problem (\ref{m1})-(\ref{mrd}) with $ u^{inc}= \nabla \Phi_p(x,z) $ satisfies the following integral equation
\begin{equation}\label{Lip}
v(x;z)=V^{inc}(x;z)-\omega^2 \int_{D\setminus \overline{D_b}} \left[ 1-\rho(y)\right] G(x,y)v(y;z)\mathrm{d} y,\ \ x\in \mathbb{R}^3 \setminus \overline{D_b}.
\end{equation}

Conversely, by direct calculation, it is easy to prove that (\ref{Lip}) also satisfies the scattering problem (\ref{m1})-(\ref{mrd}).

\begin{remark}
The equivalence between the problem $(\ref{m1})-(\ref{mrd})$ and Lippmann-Schwinger type equation $(\ref{Lip})$ in 
the case $D_b=\emptyset$ was proved in \cite[Lemma 5.7]{peter1998On}.
\end{remark}

Next, we give the proof of Theorem \ref{medium}.
\begin{proof}
Define the volume operator $T$ in $L^p(D\setminus \overline{D_b})$ by 
\begin{equation}
\nonumber
(T\varphi)(x):=\omega^2 \int_{D\setminus \overline{D_b}} [1-\rho(y)] G(x,y)\varphi(y;z)\mathrm{d} y,\ \ x\in D\setminus \overline{D_b}.
\end{equation}
Then $(\ref{Lip})$ equals to
\begin{equation}\label{Lip1}
(I+T)v(x;z)=V^{inc}(x;z),  \ \ x\in D\setminus \overline{D_b}.
\end{equation}
Since $z\in \mathbb{R}^3 \setminus \overline{D}$ and the problem (\ref{bdeq}) is well-posed, we know that
$V^{inc}(x;z)\in L^2(D\setminus \overline{D_b})$ . Moreover, $\rho(x)\in L^{\infty}( D\setminus \overline{D_b})$ and 
$G(x,y)\in L^2(D\setminus \overline{D_b})$ show that the volume potential operator $T$ is bounded from $L^2( D\setminus \overline{D_b})$ into $W^{2,2}( D\setminus \overline{D_b})$ (see\cite[Ch.\uppercase\expandafter{\romannumeral 4},Theorem 2.3]{Kupradze}, \cite[Theorem 5.6]{peter1998On}) and therefore compact in $L^2( D\setminus \overline{D_b})$. Then by
the uniqueness result for the problem $(\ref{m1})-(\ref{mrd})$, the operator $ I+T $ is of Fredholm type with zero index,
which means the existence of a unique solution $v(x;z)\in L^2(D\setminus \overline{D_b})$ with the estimates
\begin{equation}
||v(x;z)||_{ L^2(D\setminus \overline{D_b})}\leq C||V^{inc}(x;z)||_{ L^2(D\setminus \overline{D_b})}\leq C(||u^{inc}(x;z)||_{ L^2(D\setminus \overline{D_b})}+||u^{inc}(x;z)||_{L^{\infty}(\partial {D_b})}).
\end{equation}
Then, the Sobolev imbedding theorems show that
\begin{equation}
||v(x;z)||_{L^2(D\setminus \overline{D_b})}+||v(x;z)-u^{inc}||_{H^1(D\setminus \overline{D_b})}
\leq C(||u^{inc}(x;z)||_{ L^2(D\setminus \overline{D_b})}+||u^{inc}(x;z)||_{L^{\infty}(\partial {D_b})}) 
\end{equation}
\end{proof}

\section{The inverse scattering problem}\label{sec3}
\setcounter{equation}{0}
Now we can state the Inverse Problem, which reads: Find the shape and the position of the inclusion
$D$ (i.e. reconstruct the boundary) from the knowledge of the far-field patterns $u^\infty(\hat{x})$ (or $ v{{sc},^\infty(\hat{x})}$ for all $\hat{x}\in\mathbb{S}^2$, for incident plane wave $u^{inc}(x)=de^{ik_e^px\cdot d}$ with direction $d \in\mathbb{S}^2$, disregarding its contents. The method is base on constructing a well-posed interior transmission problem in a \emph{small} domain associated with Navier equation or modified Navier equation. The smallness of the domain is key since it ensure that the given frequency $\omega$ is not a transmission eigenvalue of the constructed interior transmission problem.

\subsection{Interior transmission problem}
Assume that $ \Omega \in\mathbb{R}^3 $ is a non-empty, open, connected and bounded set with a $ \mathit{C}^2 $ boundary $ \partial \Omega $. The \emph{ modified interior transmission problem } (MITP) is stated as:
\begin{alignat}{3}
\label{MI1}   \mu_i\Delta U(x)+(\lambda_i+\mu_i)\nabla(\nabla\cdot U(x))-U(x)=&\rho_1, \qquad \qquad &x\in &\Omega,\\
    \mu_e\Delta V(x)+(\lambda_e+\mu_e)\nabla(\nabla\cdot V(x))-V(x)=&\rho_2, \qquad \qquad  &x\in &\Omega,\\
     U(x)-V (x)= &f, \qquad \quad \ &x\in &\partial \Omega, \\
 \label{MI2}    T_e U(x)-T_e V(x)=&h,\qquad \quad \ &x\in &\partial \Omega,  
\end{alignat}
where $\rho_1,\rho_2 \in (L^2(D))^3$, $f\in (H^{1/2}(\partial D))^3$ and $h\in (H^{-1/2}(\partial D))^3$. This problem for the case of an anisotropic nonhomogeneous inclusion $D$ was studied in Charalambopoulos \cite{Charalambopoulos2002On}. Then in isotropic homogeneous case, the following result also holds (see \cite[Theorem 4]{Charalambopoulos2002On}).
\begin{lemma}\label{Mes}
If $(\mu_i- \mu_e)[(3\lambda_i+2\mu_i)-(3\lambda_e+2\mu_e)]>0$, the modified interior transmission problem $(\ref{MI1})-(\ref{MI2})$ has a unique strong solution $ (U,V) $ that satisfies the priori estimate
\begin{equation}
||U||_{H^1(D)}+||V||_{H^1(D)}\leq C(||\rho_1||_{L^2(D)}+||\rho_2||_{L^2(D)}+||f||_{H^{1/2}(\partial D)}+||h||_{H^{-1/2}(\partial D)})
\end{equation}
where the constant $C$ is independent of $\rho_1,\rho_2,f,h$, and depends only on the Lam$\acute{e}$ constants $\lambda_e, \mu_e.$
\end{lemma}

In the case $\mu_i=\mu_e=\mu>0$ and $\lambda_i=\lambda_e=\lambda$ satisfies $3\lambda+2\mu>0$, the above result does not hold because the existence of the eigenvalue in $\Omega$. Then we consider the following \emph{ interior transmission problem }(ITP): 
\begin{alignat}{3} 
 \label{ITP} \Delta^* U(x)+\omega^2\rho(x)U(x)=&0, \qquad \qquad &\texttt{in}\ &\Omega,\\
 \label{IT2P}\Delta^* V(x)+\omega^2 V(x)=&0, \qquad \qquad  &\texttt{in}\ &\Omega,\\
 \label{ITP3}  U(x)-V (x)= &f, \qquad \quad \ &\texttt{on}\ &\partial \Omega, \\
  \label{ITP4}   T_e U(x)-T_e V(x)=&h,\qquad \quad \ &\texttt{on}\  &\partial \Omega,  
\end{alignat}
where $f\in (H^{1/2}(\partial\Omega))^3$ and $h\in (H^{-1/2}(\partial \Omega))^3$, $\rho(x)\in L^{\infty}(\Omega)$ is positive real valued. More over, we define constants
\begin{equation}
\nonumber
\rho_*:= \inf_{\Omega}\rho(x),\ \ \rho^*:= \sup_{\Omega}\rho(x).
\end{equation}

In \cite{Cakoni2021the}, they consider the interior transmission eigenvalue problem for the elastic waves propagating through an anisotropic inhomogeneous media of bounded support containing an obstacle and gave some important results about the eigenvalues. Similar to the method for Helmholtz equation in \cite{Cakoni2010The,Cakoni2012Transmission}, we extend the results in \cite{Cakoni2021the} to the problem (\ref{ITP})-(\ref{ITP4}). Define the Hilbert space 
\begin{equation}\label{space}
\mathcal{H}(\Omega)=\lbrace u\in (H^1(\Omega))^3 \mid \Delta^* u\in (L^2(\Omega))^3\rbrace
\end{equation}
with norm $||u||^2_{\mathcal{H}(\Omega)}=||u||^2_{H^1(\Omega)}+||\Delta^* u||^2_{L^2(\Omega)}$. For $u\in \mathcal{H}(\Omega)$, it's easy to prove that $\gamma_0u=u|_{\partial \Omega}\in (H^{1/2}(\partial\Omega))^3$ and  $\gamma_1u=T_eu|_{\partial \Omega}\in (H^{-1/2}(\partial\Omega))^3$. In particular, if $\gamma_0u=\gamma_1u=0$ for all $u\in \mathcal{H}(\Omega)$, then $\mathcal{H}(\Omega)=(H^2_0(\Omega))^3$.

\begin{lemma}\label{ITV}
Assume that $0\neq \omega \in \mathbb{R}$ and either $\rho_*>1 $ of $\rho^*<1 $. Suppose that $(U,V)\in (H^1(\Omega))^3\times (H^1(\Omega))^3$ is a solution to problem $(\ref{ITP})-(\ref{ITP4})$. Define $u:= U-V$, then
$u\in \mathcal{H}(\Omega) $ and satisfies
\begin{equation}\label{var}
a(u,h)=0 \ \ \texttt{for all}\ \ h\in (H^2_0(\Omega))^3,
\end{equation}
where
\begin{equation}
\nonumber
\begin{split}
a(u,h)=&\int_{\Omega}(1-\rho(x))^{-1}(\Delta^*+\omega^2)u\cdot(\Delta^* +\omega^2)\overline{h}\   \mathrm{d}x,\\
&+\omega^2\int_{\Omega}\nabla u:\mathbb{C}:\nabla \overline{h}\ \mathrm{d}x 
 -\omega^4\int_{\Omega}u\cdot \overline{h}\ \mathrm{d}x.
\end{split}
\end{equation}
Here $\mathbb{C}=(C_{jkmn})_{3\times 3}$ is elastic moduli, and $C_{jkmn}=\lambda \delta_{ij}\delta_{kl}+ \mu(\delta_{ik}\delta_{jl}+\delta_{il}\delta_{jl})$.
\end{lemma}

\begin{proof}
Let $u:= U-V$, from the equation (\ref{ITP}), it follows that
\begin{equation}\label{ITP5} 
\begin{split}
V(x)&=\frac{1}{\omega^2}(1-\rho(x))^{-1}(\Delta^*+\omega^2\rho(x))u\\
&=\frac{1}{\omega^2}(1-\rho(x))^{-1}(\Delta^*+\omega^2)u -u
\end{split}
\end{equation}
Testing equation (\ref{IT2P}) by $h\in (H^2_0(\Omega))^3$, we obtain that
\begin{equation}
\nonumber
0=\int_{\Omega}(\Delta^*V+\omega^2 V)\cdot \overline{h}\mathrm{d}x,
\end{equation}
Using Green's second vector theorem and the fact that $h=T_eh=0$ on $\partial \Omega$, we have
\begin{equation}
\nonumber
0=\int_{\Omega}V\cdot(\Delta^*\overline{h} +\omega^2 \overline{h} )\ \mathrm{d}x,
\end{equation}
Combine $(\ref{ITP5})$ and above identity, we obtain
\begin{equation}
\nonumber
\begin{split}
0=&\int_{\Omega}(1-\rho(x))^{-1}(\Delta^*+\omega^2)u\cdot(\Delta^* +\omega^2)\overline{h}\   \mathrm{d}x,\\
&+\omega^2\int_{\Omega}\nabla u:\mathbb{C}:\nabla \overline{h}\ \mathrm{d}x 
 -\omega^4\int_{\Omega}u\cdot \overline{h}\ \mathrm{d}x.
\end{split}
\end{equation}
Then it is not difficult to see that $u\in \mathcal{H}(\Omega) $ and thus the lemma follows.
\end{proof}

The opposite implication is easy to get by direct calculation.
\begin{lemma}\label{ITPin}
Assume that $0\neq \omega \in \mathbb{R}$ and either $\rho_*>1 $ of $\rho^*<1 $. Suppose that $u\in \mathcal{H}(\Omega)$ satisfies $(\ref{var})$. In $\Omega$, define  
\begin{equation}
\nonumber
\begin{split}
V(x)&=\frac{1}{\omega^2}(1-\rho(x))^{-1}(\Delta^*+\omega^2\rho(x))u,\\
U(x)&=V(x)+u(x),
\end{split}
\end{equation}
then $(U,V)\in (H^1(\Omega))^3\times (H^1(\Omega))^3$ solves the problem $(\ref{ITP})-(\ref{ITP4}).$
\end{lemma}

Let us define the following auxiliary bounded sesquilinear forms on $ (H^2_0(\Omega))^3\times (H^2_0(\Omega))^3 $:
\begin{alignat*}{3}
&b(u,v);=\int_{\Omega}(\rho(x)-1)^{-1}(\Delta^*+\omega^2)u\cdot(\Delta^* +\omega^2)\overline{v}\   \mathrm{d}x+\omega^4\int_{\Omega}u\cdot \overline{v}\ \mathrm{d}x,\\
&\tilde{b}(u,v);=\int_{\Omega}\frac{\rho(x)}{1-\rho(x)}(\Delta^*+\omega^2)u\cdot(\Delta^* +\omega^2)\overline{v}\   \mathrm{d}x+\int_{\Omega}\Delta^*u\cdot \Delta^*\overline{v}\ \mathrm{d}x,\\
&c(u,v);=\int_{\Omega}\nabla u:\mathbb{C}:\nabla \overline{v}\ \mathrm{d}x .
\end{alignat*}
In terms of these operators we can rewrite (\ref{var}) as 
\begin{alignat*}{3}
&b(u,h)- \omega^2 c(u,h)=0 , \ \ &\texttt{for} \ & \rho_*>1, \forall h\in (H^2_0(\Omega))^3,\\
\texttt{or}\ &\tilde{b}(u,h)- \omega^2 c(u,h)=0 , \ \ &\texttt{for}\ & \rho^*<1, \forall h\in (H^2_0(\Omega))^3,
\end{alignat*}

In \cite{Bellis2013nature}, the author proved that the set of transmission eigenvalues affiliated with (\ref{ITP})-(\ref{ITP4})(or (\ref{var})) is countable, discrete and with infinity being the only possible accumulation point. Then, we can obtain a result in elastics similar to Lemma 2.4 in \cite{Cakoni2010The} for acoustics.
\begin{lemma}\label{coercive}
If  $\rho_*>1 $ of $\rho^*<1 $, then 
\begin{equation}
\nonumber
a(u,u)\geq C ||u||^2_{H^2_0(\Omega)}, \ \ \forall \ u\in (H^2_0(\Omega))^3 
\end{equation}
for $0<\omega^2<min\lbrace\frac{m \lambda_1(\Omega)}{\rho^*},  m \lambda_1(\Omega)\rbrace$, where $\lambda_1(D)$ is the first Dirichlet eigenvalue of the operator $-\Delta$ in $\Omega$ and $m=min\lbrace\mu,3\lambda+2\mu\rbrace$.
\end{lemma} 
 
\begin{proof}
In the case $\rho_*>1$, set $\gamma=\frac{1}{\rho^*-1}$. Then for $u\in (H^2_0(\Omega))^3$
\begin{equation}
\nonumber
\begin{split}
b(u,u)&=\int_{\Omega}(\rho(x)-1)^{-1}(\Delta^*+\omega^2)u\cdot(\Delta^* +\omega^2)\overline{u}\   \mathrm{d}x+\omega^4\int_{\Omega}u\cdot \overline{u}\ \mathrm{d}x\\
&\geq \gamma ||(\Delta^*+\omega^2)u||^2_{L^2(\Omega)}+\omega^4||u||^2_{L^2(\Omega)}\\
&\geq \gamma ||\Delta^*u||^2_{L^2(\Omega)}- 2\gamma \omega^2 ||\Delta^*u||^2_{L^2(\Omega)}||u||^2_{L^2(\Omega)}+(\gamma +1)\omega^4||u||^2_{L^2(\Omega)}\\
&\geq \gamma (1-\frac{\gamma}{\epsilon})||\Delta^*u||^2_{L^2(\Omega)}+(\gamma +1-\epsilon)\omega^4||u||^2_{L^2(\Omega)}
\end{split}
\end{equation}
Here we use the following inequality(see \cite{Cakoni2012Transmission}), for each $\epsilon>0$
\begin{equation}
\nonumber
\begin{split}
&\gamma X^2- 2\gamma XY+(\gamma+1)Y^2\\
=&\epsilon(Y-\frac{\gamma}{\epsilon}X)^2+\gamma(1-\frac{\gamma}{\epsilon})X^2+(\gamma +1-\epsilon)Y^2\\
\geq &\gamma (1-\frac{\gamma}{\epsilon})X^2+(\gamma +1-\epsilon)Y^2.
\end{split}
\end{equation}
Choosing $\gamma<\epsilon<\gamma+1$ immediately shows that $b$ is coercive. Moreover, since $ (\mathbb{C}:\nabla u)\in (H_0^1(\Omega))^3$, by $Poincar\acute{e}$ inequality, we have
\begin{equation}
\nonumber
\begin{split}
c(u,u)&=\int_{\Omega}\nabla u:\mathbb{C}:\nabla \overline{u}\ \mathrm{d}x \\
&\leq \frac{1}{m} ||\mathbb{C}:\nabla u||^2_{L^2(\Omega)}\\
&\leq \frac{1}{m \lambda_1(\Omega)}||\nabla \cdot(\mathbb{C}:\nabla u)||^2_{L^2(\Omega)}\\
&= \frac{1}{m \lambda_1(\Omega)}||\Delta^*u||^2_{L^2(\Omega)}
\end{split}
\end{equation}
where $m=min\lbrace\mu,3\lambda+2\mu\rbrace, \lambda_1(D)$ is the first Dirichlet eigenvalue of the operator $-\Delta$ in $\Omega$.
Then, combine above inequalities, we have
\begin{equation}
a(u,u)\geq (\gamma -\frac{\gamma^2}{\epsilon}-\frac{\omega^2}{m \lambda_1(\Omega)})||\Delta^*u||^2_{L^2(\Omega)}+(\gamma +1-\epsilon)\omega^4||u||^2_{L^2(\Omega)}.
\end{equation}
Choosing $\epsilon $ arbitrary closed to $\gamma+1$, and $\omega^2<(\gamma -\frac{\gamma^2}{\epsilon})m \lambda_1(\Omega)<\frac{m \lambda_1(\Omega)}{\rho^*}$, then $a(u,u)$ is coercive which means
\begin{equation}
\nonumber
a(u,u)\geq C ||u||^2_{H^2_0(\Omega)}, \ \ \forall \ u\in (H^2_0(\Omega))^3.
\end{equation}
The same for the case $\rho^*<1$, we obtain that $a(u,u)$ is coercive when $\omega^2< m \lambda_1(\Omega)$.
\end{proof}

Next, we give an estimate of problem $(\ref{MI1})-(\ref{MI2})$. We assume that the data  $f\in (H^{1/2}(\partial\Omega))^3$ and $h\in (H^{-1/2}(\partial \Omega))^3$ satisfy the condition $(\mathbf{C})$ with some $u_0 \in \mathcal{H}(\Omega)$ such that  $\gamma_0u_0=f$ and  $\gamma_1u_0=h$. 
\begin{lemma}\label{ITPwell}
Assume that $f,h$ satisfy the condition $(\mathbf{C})$ with $u_0 \in \mathcal{H}(\Omega)$. For any fixed $\omega\neq 0$, if the diameter of $ \Omega $ is small enough such that $\omega^2<min\lbrace\frac{m \lambda_1(\Omega)}{\rho^*},\lambda_1(\Omega)\rbrace$, then the interior transmission problem $(\ref{MI1})-(\ref{MI2})$ has a unique solution $(U,V)\in (L^2(\Omega))^3\times (L^2(\Omega))^3$ with
\begin{equation}
||U||_{L^2(\Omega)}+||V||_{L^2(\Omega)}\leq C||u_0||_{\mathcal{H}(\Omega)}.
\end{equation}
\begin{proof}
For any fixed $\omega\neq 0$, if the diameter of $ \Omega $ is small enough such that $\omega^2<min\lbrace\frac{m \lambda_1(\Omega)}{\rho^*},\lambda_1(\Omega)\rbrace$, then define $u:=U-V \in \mathcal{H}(\Omega) , \tilde{u}=u-u_0 \in (H^2_0(\Omega))^3$.
By Lemma \ref{coercive}, we have
\begin{equation}
\nonumber
a(\tilde{u},\tilde{u})\geq C ||\tilde{u}||^2_{H^2_0(\Omega)},
\end{equation}
Since $u$ satisfied the variational problem (\ref{var}) which means
\begin{equation}
\nonumber
a(u,\tilde{u})=a(\tilde{u},\tilde{u})+a(u_0,\tilde{u})=0.
\end{equation}
Since $a(u,u)$ is bounded and coercive on $ (H^2_0(\Omega))^3\times (H^2_0(\Omega))^3 $, by Lax-Milgram theorem, the variational problem (\ref{var}) has a unique solution $ \tilde{u}\in H^2_0(\Omega) $ satisfying
\begin{equation}
\nonumber
 ||\tilde{u}||^2_{H^2_0(\Omega)}\leq C  ||u_0||^2_{\mathcal{H}(\Omega)}.
\end{equation}
Then by Lemma \ref{ITPin}, define $V(x):=\frac{1}{\omega^2}(1-\rho(x))^{-1}(\Delta*+\omega^2\rho(x))u,
U(x):=V(x)+u(x)$, it's easy to derive
\begin{equation}
\nonumber
||U||_{L^2(\Omega)}+||V||_{L^2(\Omega)}\leq C||u_0||_{\mathcal{H}(\Omega)}
\end{equation}
from the fact that $u=u_0+\tilde{u}.$
\end{proof}
\end{lemma}
\subsection{Uniqueness of the Inverse Problem }
Based on above results, we will prove the uniqueness results in determining the inhomogeneous penetrable obstacle
$D$ disregarding its contents $ \rho $ and $ D_b $. 

First, we give the following uniqueness result in the case that medium in $D_i$ and $D_e$ are different.
\begin{theorem}\label{main}
  Given $\omega>0$, suppose $(\mu_i- \mu_e)[(3\lambda_i+2\mu_i)-(3\lambda_e+2\mu_e)]>0$ and let $u^\infty(\hat{x},d)$ 
  and $ \widetilde{u}^\infty(\hat{x},d)$ be the far-field patterns of the solutions $u(x)$ and $\widetilde{u}(x)$ to the
  transmission problem $(\ref{eq1})-(\ref{rd})$ with respect to the scatterers $(D,\rho,D_b)$ and $( \widetilde{D},  \widetilde{\rho}, \widetilde{D_b})$. Then if $u^\infty(\hat{x},d)= \widetilde{u}^\infty(\hat{x},d)$
   for all $\hat{x},d \in \mathbb{S}^2$, then $D=\widetilde{D}$.
\end{theorem}

\begin{proof}
  Assume that $D\neq \widetilde{D}$. Without loss of generality, choose $z^*\in \partial D\setminus \partial {D_b}$ and define
 \begin{equation}
  \nonumber
 z_j:=z^*+ \frac{\delta}{j} n(z^*), \qquad j=1,2,3,...  
  \end{equation}  
with a small enough $\delta>0$ such that $z_j\in B$, where $B$ denotes a small ball centered at $z^*$
such that $B\cap(\tilde{D}\cup D_b)= \emptyset.$

\begin{figure}[htbp]
\centering
\includegraphics[scale=0.5]{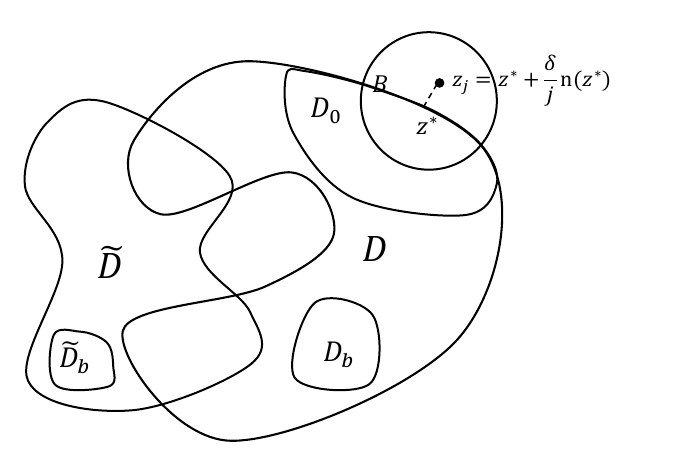}
\caption{Two different scattering obstacles.}\label{f1}
\end{figure}

Consider the transmission problem (\ref{eq1})-(\ref{rd}) with the boundary data $f$ and $h$ induced by
the longitudinal part of the incident point source
\begin{equation}
\nonumber
 u_j^{inc}=\frac{\Gamma^p(x,z_j)\cdot q}{||\nabla_x\nabla_x^{\mathsf{T}}\Phi_p(x,z_j)\cdot q||_{L^2(\partial D)}}=
   \frac{ {k_p}^2}{\omega^2}\frac{\nabla_x\nabla_x^{\mathsf{T}}\Phi_p(x,z_j)\cdot q}{{||\nabla_x\nabla_x^{\mathsf{T}}\Phi_p(x,z_j)\cdot q||_{L^2(\partial D)}}},\qquad j=1,2,3,...
\end{equation}
for $q\in \mathbb{R}^3$, where $\Phi_p(x,z_j)$ is the fundamental solution to the three-dimensional
Helmholtz equation given by
\begin{equation}
\nonumber
\Phi_p(x,z_j)= \frac{1}{4\pi}\frac{e^{ik_e^p|x-z_j|}}{|x-z_j|},\qquad x\neq z_j.
\end{equation} 
As assumption $u^\infty(\hat{x},d)=\tilde{u}^\infty(\hat{x},d)$ for all $\hat{x},d \in \mathbb{S}^2$, we can use
 Rellich's lemma \cite[Theorem 2]{Peter1993A} to get
\begin{equation}\label{in1}
 u(x,d)=\tilde{u}(x,d) \qquad \text{ in}\  \overline{G}
\end{equation}
for all $d \in \mathbb{S}^2$, where $G$ denotes the unbounded component of $\mathbb{R}^3\setminus\overline{ (D\cup\tilde{D})} $.

Then we choose a smooth and bounded domain $B$ such that $\mathbb{R}^3\setminus B$ is connected,
 $\overline{(D\cup\tilde{D})}\subset B$ and $z_j\notin \overline{B}$. From the classical result of denseness of the Herglotz wave operator \cite{Colton2013}, we can find a sequence of densities $(g_{j,m})_n\in L^2(\mathbb{S}^2)$ such that Herglotz wave function $H(g_{j,m})\rightarrow \Phi_p(x,z_j)$ in $L^2(\partial B)$. We define
   $g^1_{j,m}(\theta)=\theta^{\mathsf T}\theta g_{j,m}(\theta), \theta\in\mathbb{S}^2$, then it is clear
   that $u^{inc}_{j,m}=H(g^1_{j,m})\cdot q$ tends to $\nabla_x\nabla_x^{\mathsf{T}}\Gamma(x,z_j)\cdot q$
   in $(L^2(\partial B))^{3\times3}$. That is , for each fixed $j\in \mathbb{N}$ there is a sequence $(u^{inc}_{j,m})$ such that
\begin{equation}\label{in2}
   u^{inc}_{j,m}\rightarrow\frac{\Gamma^p(x,z_j)\cdot q}{||\nabla_x\nabla_x^{\mathsf{T}}\Phi_p(x,z_j)\cdot q||_{L^2(\partial D)}},\qquad \nabla  u^{inc}_{j,m}\rightarrow \nabla(\frac{\Gamma^p(x,z_j)\cdot q}{||\nabla_x\nabla_x^{\mathsf{T}}\Phi_p(x,z_j)\cdot q||_{L^2(\partial D)}}),\quad m\rightarrow\infty
\end{equation}
uniformly on $\overline{(D\cup\tilde{D})}$. Also, we denote $\Gamma_j(x, z_j)$ as $\nabla_x\nabla_x^{\mathsf{T}}\Gamma(x,z_j)\cdot q$.
 Since the $u^{inc}_{j,m}$ are linear combinations of plane waves, the corresponding scattered wave $u_{j,m}$ and $\tilde{u}_{j,m}$
  for the complex scatterers $(D,\rho,D_b)$ and $(\tilde{D},\tilde{\rho},\tilde{D}_b)$ in $G$, so we have
\begin{equation}\label{in3}
 \begin{split}
   u^{inc}_{j,m}+{u}_{j,m}- {v}_{j,m} &=0, \qquad T_e u^{inc}_{j,m}+T_e { u}_{j,m}-T_i {v}_ {j,m}=0,\quad \text{on} \ \partial D \cap  \partial G,\\
 u^{inc}_{j,m}+ \widetilde{u}_{j,m}- \widetilde{v}_{j,m} &=0, \qquad T_e u^{inc}_{j,m}+T_e \widetilde{u}_{j,m}-T_i \widetilde{v}_ {j,m}=0,\quad \text{on}\  \partial \widetilde{D} \cap \partial G,
 \end{split}
\end{equation}
where $({v}_ {j,m}, {u}_{j,m})$ and $(\tilde{{v}}_ {j,m}, \tilde{{u}}_{j,m})$ are the unique solution to the
transmission problem (\ref{eq1})-(\ref{rd}) with respect to the complex scatterers $(D,\rho,D_b)$ and $(\widetilde{D},\widetilde{\rho},\widetilde{D}_b)$ with incident wave $u^{inc}_{j,m}$. Noting that
 \begin{equation}\label{in4}
 \begin{split}
  u^{inc}_{j}+{u}_{j }- {v}_{j } &=0, \qquad T_e u^{inc}_{j }+T_e { u}_{j }-T_i {v}_ {j }=0,\quad \text{on} \ \partial D  ,\\
 u^{inc}_{j}+ \widetilde{u}_{j }- \widetilde{v}_{j } &=0, \qquad T_e u^{inc}_{j }+T_e \widetilde{u}_{j }-T_i \widetilde{v}_ {j }=0,\quad \text{on}\  \partial \widetilde{D}  ,
 \end{split}
\end{equation}
and by the well-posedness of the transmission problem (\ref{eq1})-(\ref{rd}), the transmission conditions
(\ref{in3}) and the convergence
 (\ref{in2}), we obtain that for each $j\in \mathbb{N}$,
\begin{equation}
\nonumber
  u_{j,m}\rightarrow u_j,\qquad \tilde{u}_{j,m}\rightarrow \tilde{u}_j,  \qquad m\rightarrow\infty
\end{equation}
uniformly on compact subsets of $G$. Thus 
\begin{equation}\label{equal}
 u_j=\tilde{u}_j,\qquad \texttt{in}\quad \overline{G},\quad j=1,2,3,...
\end{equation}

Since $z^*\in \partial D\setminus \partial {D_b}$ and $\partial D\in C^2,$ there is a small
smooth $C^2$ domain
 $D_0$ such that $B\cap D\subset D_0\subset {D\setminus\overline{(  \widetilde{D} \cup D_b)}}$.
Define $U_j= v_j, V_j=\widetilde{u}_j+u_j^{inc}$ in $D_0$. Then $(U_j,V_j)$ satisfies the modified
  interior transmission problem as follows:
\begin{alignat}{3}\label{MITP}
   \mu_i\Delta U_j(x)+(\lambda_i+\mu_i)\nabla(\nabla\cdot U_j(x))-U_j(x)=&\rho_1, \qquad \qquad &x\in & D_0,\\
    \mu_e\Delta V_j(x)+(\lambda_e+\mu_e)\nabla(\nabla\cdot V_j(x))-V_j(x)=&\rho_2, \qquad \qquad  &x\in &D_0,\\
     U_j(x)-V_j(x)= &f_1, \qquad \quad \ &x\in &\partial D_0 , \\
     T_i U_j(x)-T_e V_j(x)=&h_1,\qquad \quad \ &x\in &\partial D_0 ,  
\end{alignat} 
with
\begin{align*}
  \rho_1(j)&:= -(\rho(x)\omega^2+1)v_j|_{D_0},\  &\rho_2(j)&:=  -( \omega^2+1) (\widetilde{u }_j+u_j^{inc})|_{D_0} , \\
  f_1(j)&:=(v_j-{\widetilde{u}}_j-u_j^{inc})|_{\partial D_0},\  &h_1(j)&:= (T_i v_j-T_e {\widetilde{u}}_j-T_e u_j^{inc})|_{\partial D_0}.
\end{align*}
From the result $u_j=\tilde{u}_j$ in $\overline{G}$, it is clear that
$f_1(j)=f_2(j)$ on $\Gamma_1:=\partial D_0\cap\partial D$. Since $z^*$ has a positive distance
 from $ \widetilde{D}$, the well-posedness of the transmission problem (\ref{eq1})-(\ref{rd}) implies that
\begin{equation}\label{uniform1}
  ||\widetilde{u}_j||_{H^2(D_0)}\leq C \quad \text{uniformly for}\  j\in \mathbb{N}.
\end{equation}
This means that $\rho_2(j)\in L^2(D_0)$ is uniformly bounded for $j\in \mathbb{N}$. From Corollary \ref{coro}, $\rho_1(j)$
 is uniformly bounded in $  L^2(D_0)$ for $j\in \mathbb{N}$.

We now prove the boundedness of $f_1(j)$ and $h_1(j)$. In this end, we define $W_j=v_j-{\widetilde{u}}_j-u_j^{inc}$.
Then $W_j$ satisfies
\begin{equation}
\nonumber
\Delta_e^* W_j= g_j,\ \ x\in {D\setminus\overline{(  \widetilde{D} \cup D_b)}},\ \ W_j|_{\Gamma_1}=0,
\end{equation} 
where $g_j:=\omega^2({\widetilde{u}}_j+u_j^{inc})-\Delta_e^*v_j \in L^2({D\setminus\overline{( \widetilde{D} \cup D_b)}})$.
 Since Lam\'{e} operator $\Delta^*$ is strong elliptic and $z_j\in \mathbb{R}^3\setminus D$, with help of a local boundary
estimate in \cite{Gilbarg2001} we can obtain that $W_j\in H^2(D\setminus\overline{(  \widetilde{D} \cup D_b)})$ which satisfies
\begin{equation}\label{elliptices}
  ||W_j||_{ H^2(D_0)}\leq C(||W_j||_{L^2({D\setminus\overline{(  \widetilde{D} \cup D_b)}})}+||g_j||_{L^2({D\setminus\overline{(  \widetilde{D} \cup D_b)}})})\leq C'
\end{equation}
uniformly for $j\in \mathbb{N}.$ Since
\begin{equation}
\nonumber
f_1(j):= W_j|_{\partial D_0},\qquad  h_1(j):= [T_e W_j+(T_i-T_e) v_j ]|_{\partial D_0},
\end{equation}
by using (\ref{uniform1}), (\ref{elliptices}) and the fact that $f_1(j)|_{\Gamma_1}=h_1(j)|_{\Gamma_1}=0$,
it's easy to get that $f_1(j)$ and $h_1(j)$ are uniformly bounded in $H^{\frac{1}{2}}(\partial D_0)$
and $H^{-\frac{1}{2}}(\partial D_0)$, respectively, for $j\in \mathbb{N}.$ Thus, by well-posedness of
modified interior transmission problem (\ref{MITP}), Lemma \ref{Mes} shows that
\begin{equation}
\nonumber
||U_j||_{H^1{(D_0)}}+||V_j||_{H^1{(D_0)}}\leq C(||\rho_{1,j}||_{L^2{(D_0)}} +||\rho_{2,j}||_{L^2{(D_0)}}+||f_1||_{H^{(\frac{1}{2}}{(\partial D_0)}}+||h_1||_{H^{-\frac{1}{2}}{(\partial D_0)}})
\end{equation}
which means that
\begin{equation}\label{contra}
 ||u_j^{  inc }||_{H^1{(D_0)}}-||\widetilde{u}_j||_{H^1{(D_0)}}\leq ||V_j||_{H^1{(D_0)}}\leq C
\end{equation}
However, since $u_j^{  inc }=\frac{\Gamma^p(x,z_j)\cdot q}{||\nabla_x\nabla_x^{\mathsf{T}}\Phi_p(x,z_j)\cdot q||_{L^2(\partial D)}}$,
it's obviously that $||u_j^{  inc }||_{H^1{(D_0)}}\rightarrow\infty$ as $ j\rightarrow \infty$. Then (\ref{contra})
 contradicts with the fact that $||\widetilde{u}_j||_{H^1{(D_0)}}$ is uniformly bounded. Hence, $D=\widetilde{D}$.

\end{proof}

If medium in $D_i$ and $D_e$ are the same, that is $\mu_i=\mu_e=\mu>0$, $\lambda_i=\lambda_e=\lambda$ satisfying $3\lambda+2\mu>0$ and $\rho(x)$, we can also obtain the uniqueness result for Inverse Problem of (\ref{m1})-(\ref{mrd}). But, the density $\rho (x)$ needs to satisfy the following Assumption  $\mathbf{1}$ (see \cite{Yang2017Uniqueness}).
 \begin{assumption}
 There exists an open neighbourhood of $\partial D$, $ \mathcal{O}\Subset D \setminus \overline{D_b}$, and a positive 
 constant $\varepsilon_0>0$ such that $|\rho (x)-1|\geq \varepsilon_0$ for a..e. $x\in \mathit{O}.$
 \end{assumption}

Next, we give the following uniqueness result. 
\begin{theorem}\label{main1} 
  Given $\omega>0$, let $u^{{sc},\infty}(\hat{x},d)$ 
  and $ \widetilde{u}^{{sc},\infty}(\hat{x},d)$ be the far-field patterns to the solutions $u^{sc}(x)$ and 
  $  \widetilde{u}^{sc}(x) $ to the scattering problem $ (\ref{m1})-(\ref{mrd}) $ with respect to scatterers $(D,\rho,D_b)$ and $( \widetilde{D},  \widetilde{\rho}, 
  \widetilde{D_b})$. Then if $u^{{sc},\infty}(\hat{x},d)= \widetilde{u}^{{sc},\infty}(\hat{x},d)$
   for all $\hat{x},d \in \mathbb{S}^2$, then $D=\widetilde{D}$.
\end{theorem}
\begin{proof}
  Assume that $D\neq \widetilde{D}$. Without loss of generality, choose $z^*\in \partial D\setminus \partial {D_b}$ and define
 \begin{equation}
  \nonumber
 z_j:=z^*+ \frac{\delta}{j} n(z^*), \qquad j=1,2,3,...  
  \end{equation}  
with a small enough $\delta>0$ such that $z_j\in B$, where $B$ denotes a small ball centered at $z^*$
such that $B\cap(\tilde{D}\cup D_b)= \emptyset.$

Consider the transmission problem (\ref{m1})-(\ref{mrd}) with the boundary data $f$ and $h$ induced by
the incident point source
\begin{equation}
\nonumber
 u_j^{inc}=\nabla_x \Phi_p(x,z_j),\qquad j=1,2,3,...
\end{equation}
where $\Phi_p(x,z_j)$ is the fundamental solution to the three-dimensional Helmholtz equation given by
\begin{equation}
\nonumber
\Phi_p(x,z_j)= \frac{1}{4\pi}\frac{e^{ik_e^p|x-z_j|}}{|x-z_j|},\qquad x\neq z_j.
\end{equation} 
Let $u_j$ and $\tilde{u}_j$ be the unique solution to the scattering problem (\ref{m1})-(\ref{mrd}) with respect to
scatterers $(D,\rho,D_b)$ and $( \widetilde{D},  \widetilde{\rho},\widetilde{D_b})$, respectively, corresponding to the incident wave $u^{inc}(x)=u_j^{inc}(x)$. As assumption $u^{sc,\infty}(\hat{x},d)=\tilde{u}^{sc,\infty}(\hat{x},d)$ for all $\hat{x},d \in \mathbb{S}^2$, we can use Rellich's lemma \cite[Theorem 2]{Peter1993A} to get
\begin{equation}
 u^{sc}(x,d)=\tilde{u}^{sc}(x,d) \qquad \text{ in}\  \overline{G}
\end{equation}
for all $d \in \mathbb{S}^2$, where $G$ denotes the unbounded component of $\mathbb{R}^3\setminus\overline{ (D\cup\tilde{D})} $. Similarly as in the proof of (\ref{equal}) in Theorem \ref{main}, we obtain
\begin{equation} 
 u^{sc}_j(x)=\tilde{u}^{sc}_j(x),\qquad \texttt{in}\quad \overline{G},\quad j=1,2,3,...
\end{equation}

Since $z^*\in \partial D\setminus \partial {D_b}$ and $\partial D\in C^2,$ there is a small
smooth $C^2$ domain $D_0$ such that $B\cap D\subset D_0\subset {D\setminus\overline{(\widetilde{D} \cup D_b)}}$.
Define $U_j= u_j, V_j=\widetilde{u}_j$ in $D_0$. Then $(U_j,V_j)$ satisfies the interior transmission problem as follows:
\begin{alignat}{3} 
   \Delta^* U_j(x)+\omega^2\rho(x)U_j(x)=&0, \qquad \qquad &\texttt{in}\ &D_0,\\
 \Delta^* V_j(x)+\omega^2 V_j(x)=&0, \qquad \qquad  &\texttt{in}\ &D_0,\\
   U_j(x)-V_j(x)= &f_j, \qquad \quad \ &\texttt{on}\ &\partial D_0, \\
    T_e U_j(x)-T_e V_j(x)=&h_j,\qquad \quad \ &\texttt{on}\  &\partial D_0,  
\end{alignat}
with
\begin{equation}
\nonumber
 f_j:=(u_j-{\widetilde{u}}_j)|_{\partial D_0},\  h_j:= (T_e u_j-T_e {\widetilde{u}}_j)|_{\partial D_0}.
\end{equation}
From the result $u^{sc}_j=\tilde{u}^{sc}_j$ in $\overline{G}$, it is clear that
$f_j=h_j$ on $\Gamma_1:=\partial D_0\cap\partial D$.  

Next, we prove that $f_j, h_j$ satisfy the condition $(\mathbf{C})$ (see Lemma \ref{ITPwell}). To this end, we choose 
a cut-off function $\chi(x)\in C_0^{\infty}(\mathbb{R}^3)$ such that 
\begin{equation}
\nonumber
\chi(x)=\begin{cases}
0,\qquad & x\in \mathbb{R}^3\setminus B,\\
1,\qquad & x\in  B_1,
\end{cases}
\end{equation}
where $B_1$ is a small ball centered at $ z^* $ satisfying that $B_1\subsetneq B$. Define the function
\begin{equation}
\nonumber
u_{0j}(x):=[1-\chi(x)](u_j-{\widetilde{u}}_j)(x).
\end{equation}
Then from well-posedness of the problem (\ref{m1})-(\ref{mrd}), we know that $u_{0j}\in  \mathcal{H}(D_0)$ for all
$j\in \mathbb{N}$ with $u_{0j} \in \mathcal{H}(\Omega)$ such that  $u_{0j}|_{\partial D_0}=f_j$ and  $T_eu_{0j}|_{\partial D_0}=h_j$. 

We now prove $ ||u_{0j}||_{\mathcal{H}(D_0)} $ is uniformly bounded for $j\in \mathbb{N}$. Since $z^*$ has a positive 
distance from $ \tilde{D} $ and scattering problem (\ref{m1})-(\ref{mrd}) is well-posed, we can get
\begin{equation}
||\widetilde{u}^{sc}_j||_{\mathcal{H}(D_0)}\leq C
\end{equation}
uniformly for $j\in \mathbb{N}$. Also from Theorem \ref{medium}, we have
\begin{equation}
||{u}^{sc}_j||_{{H}^1(D_0)}\leq C
\end{equation}
uniformly for $j\in \mathbb{N}$. Combine above two equalities, we have
\begin{equation}\label{uniform}
||u_{0j}(x)||_{{H}^1(D_0)}\leq ||{u}^{sc}_j-\widetilde{u}^{sc}_j||_{{H}^1(D_0)}\leq C.
\end{equation}
Then, we need to prove that $\Delta^*u_{0j}$ is uniformly bounded in $(L^2(D_0)^3$. By direct calculation, it is 
found that
\begin{equation} 
\begin{split}
 \Delta^*u_{0j}&=a(x)\Delta^*b(x)+(\lambda+3\mu)\nabla a(x)\cdot\nabla b(x)^{\mathsf T}+
  (\mu\Delta a(x)\\
  &+(\lambda+\mu)\nabla(\nabla a(x))^{\mathsf T})\cdot b(x)^{\mathsf T}+(\lambda+\mu)(\nabla\cdot b(x))\nabla a(x) 
\end{split}
\end{equation}
where $a(x):=1-\chi(x), b(x):=u_j(x)-\widetilde{u}_j(x)$. Since $ || u_j^{inc}||_{{L}^2(D_0)\setminus \overline{B_1}}\leq C$, it is easy to derive that
\begin{equation} 
\begin{split}
\Delta^*(u_j-\widetilde{u}_j)=\omega^2({u}^{sc}_j-\widetilde{u}^{sc}_j)+(1-\rho(x))\omega^2 {u}_j 
\end{split}
\end{equation}
is uniformly bounded in $ {L}^2(D_0\setminus \overline{B_1}) $ for $j\in \mathbb{N}$ from (\ref{uniform}) and Theorem \ref{medium}. This, together with $\chi(x)\in C_0^{\infty}(\mathbb{R}^3)$ and $ \chi(x)|_{B_1}=0 $ , it is easy to get 
$ ||\Delta^*u_{0j}||_{{L}^2(D_0)}\leq C$ uniformly for $j\in \mathbb{N}$. Then $ u_{0j} $ is uniformly bounded in $\mathcal{H}(D_0)$. Thus $f_j, h_j$ satisfy the condition $(\mathbf{C})$ with $ u_{0j}\in \mathcal{H}(D_0) $ for all
 $j\in \mathbb{N}$.
  
Since the smallest Dirichlet eigenvalue $\lambda_1(D_0)$ of $-\bigtriangleup$ in $D_0$ tends to $+\infty$ as the diameter
of $D_0$ goes to zero. Then we can choose $D_0$ small enough such that  $0<\omega^2<min\lbrace\frac{m \lambda_1(\Omega)}{\rho^*},  m \lambda_1(\Omega)\rbrace$. This together with Lemma \ref{ITPwell} shows that 
$||V_j||_{{L}^2(D_0)}=||\widetilde{u}_j||_{{L}^2(D_0)}\leq C$  uniformly for $j\in \mathbb{N}$. Then
\begin{equation}\label{contract}
||\nabla_x \Phi_p(x,z_j)||_{{L}^2(D_0)}-||\widetilde{u}^{sc}_j||_{{L}^2(D_0)}=||\widetilde{u}_j||_{{L}^2(D_0)}\leq C
\end{equation}
 uniformly for $j\in \mathbb{N}$. But 
 \begin{equation}
 \nonumber
 \begin{split}
 ||\nabla_x \Phi_p(x,z_j)||^2_{{L}^2(D_0)}&=\int_{D_0}|(\frac{1}{|x-z_j|
}-ik_p)\frac{e^{ik_p|x-z_j|}}{4\pi|x-z_j|}\cdot \frac{x-z_j}{|x-z_j|}|^2 \ \mathrm{d}x\\
&\geq (4\pi)^{-2}\int_{D_0}\frac{1}{|x-z_j|^4}\ \mathrm{d}x\\
&=\textit{O}(j).
\end{split}
\end{equation}
This is a contradiction to (\ref{contract}) when $j\rightarrow \infty$. Then, $D=\tilde{D}$.
\end{proof}

\begin{remark}(i) Theorem \ref{main} and Theorem \ref{main1} are right when the incident plane waves $u^{inc}(x)=d e^{ik_e^s d\cdot x},\quad d\in  \mathbb{S}^2$; and for objects $D_b $ with Neumann boundary condition $(T_iv=0$ on $\partial D_b)$, the Theorems are also right.\\
(ii) Besides, it's easy to get the unique determination of obstacle in two-dimension case in the same way.
\end{remark}

\appendix
 \renewcommand{\appendixname}{Appendix~\Alph{section}}
\section{System of singular integral equations}

We sketch the theory and introduce some important results; see  for details. Consider a linear singular integral system as
\begin{equation}\label{A.1}
 \textbf{L}(u)(x)=A(x)u(x)+\int_{S}B(x,x-y)u(y){\rm d}S(y)+\textbf{T}(u)(x)=F(x)
\end{equation}
where $A(x)=||a_{ij}(x)||_{3\times3}\in L^p(S)$, $B(x,y)=||b_{ij}(x,y)||_{3\times3}$ is the matrix of singular kernels and $\textbf{T}$ is the matrix of compact (weakly-singular) operators on smooth surface $S$. Then
\begin{equation}\label{A.2}
  \chi(x,\theta)=A(x)+B(x,\frac{x-y}{|x-y|}), \ \ (x,y)\in S\times S
\end{equation}
will be called the characteristic(matrix) of $\textbf{L}$. $\theta$ is the argument of the point $\frac{x-y}{|x-y|}$ with
$\theta\in (-\pi,\pi].$

Expand the characteristic in Fourier series
\begin{equation}
\nonumber
\chi(x,\theta)=\Sigma_{-\infty}^{+\infty} A_n(x)e^{in\theta} .
\end{equation} 

Obviously, $A_0(x)=A(x).$ Then, the symbol(matrix) of $\textbf{L}$, $\sigma(\textbf{L})$, is defined by
\begin{equation}\label{A5}
  \sigma(\textbf{L})=\Sigma_{-\infty}^{+\infty} A^*_0(x)e^{in\theta}
\end{equation}
where
\begin{equation}
\nonumber
 A^*_n(x)=\frac{2\pi i^{|n|}}{|n|}A_n(x) \ \text{for}\  n\neq0,\  \text{and}\  A^*_n(x)=A_0(x)=A(x).
\end{equation}

As compact terms in $\textbf{L}$ do no contribution to $\sigma(\textbf{L})$. The following results can be found in Mikhlin \cite[Theorem 2.34-Theorem 2.36] {Mikhlin1965}.
\begin{theorem}\label{th5.1}
  Suppose operator $\textbf{L}$ is bounded in $L^{p}(S)$, and its symbol satisfies
\begin{equation}
  \nonumber
   \mathrm{inf}| \mathrm{det}\ \sigma(\textbf{L})|>0,
  \end{equation} 
  then $\textbf{L}$ is of normal (or regular) type.
\end{theorem}
If operator $\textbf{L}$ is a normal operator from $L^{p}(S)$ into $L^{p}(S)$, $p>1$, then its adjoint
 operator $\textbf{L}^*$ is also a normal operator from $L^{p'}(S)$ into $L^{p'}(S)$,
 $\frac{1}{p}+\frac{1}{p'}=1.$ The index of $\textbf{L}$ is defined as
\begin{equation}
 \nonumber
 \mathrm{ind}\textbf{L}= \mathrm{dim}\mathrm{N}(\textbf{L})-\mathrm{dim}\mathrm{N}(\textbf{L}^*)
 \end{equation}
where $\mathrm{N}(\textbf{L})$ and $\mathrm{N}(\textbf{L}^*)$ are zero spaces of the operators $\textbf{L}$
and $\textbf{L}^*$, respectively.

  Consider the operator
 \begin{equation}
  \nonumber
   \textbf{L}_\lambda(u)(x)=A(x)u(x)+\lambda\int_{S}B(x,x-y)u(y){\rm d}S(y)
  \end{equation} 
where $\lambda$ is a complex number, while $A,B,u$ and $S$ are the same as above. Then we have a result about its index\cite[Ch.\uppercase\expandafter{\romannumeral 4}, Theorem 6.1-6.7]{Kupradze}.
\begin{theorem}\label{th5.2}
If $\Gamma$ is a continuous curve on the complex plane connecting the origin of coordinates with the point $\lambda_0$ and $\textbf{L}$ is a normal type operator for any $\lambda\in \Gamma$, then
\begin{equation}
\nonumber
\mathrm{ind}\mathbf{L}_{\lambda_0}=0.
\end{equation}
\end{theorem}
The main result of the theory is that if $\mathrm{ind}\mathbf{L}=0$, then the Fredholm theorems hold, which
means system(\ref{A.1}) has a unique solution.

\section{Symbolic}
In this section, we will prove that the system (\ref{ind}) in Section 2 is normally solvable and Fredholm theorems hold for it.
Consider the problem
\begin{align*}
 \mu_i\Delta v(x)+(\lambda_i+\mu_i)\nabla(\nabla\cdot v(x))+ \rho\omega^2v(x)=&0, &&x\in D_i,\\
 \mu_e\Delta u(x)+(\lambda_e+\mu_e)\nabla(\nabla\cdot u(x))+ \omega^2 u(x)=&0,  &&x\in D_e,\\
 v(x)-u(x)=&F(x),    &&x\in \partial D , \\
 T^{\gamma_i}_iv(x)-T^{\gamma_e}_eu(x)=&H(x), &&x\in \partial D , \\
 v(x)= &0,  &&x\in\partial D_b,\\
 \lim_{r\rightarrow \infty }r[\frac{\partial u_\alpha(x)}{\partial r}-i k_e^\alpha u_\alpha(x)]=&0,&& \alpha=p,s,
\end{align*}
we seek solutions of problem in the form
\begin{equation}\label{app1}
\begin{split}
    v(x)=&\int_{\partial D_b}T^{\kappa_i}_{i,y}[\Gamma_i(x,y)]^{\mathsf T}\eta(y){\rm d}s(y)+ \int_{\partial D}\Gamma_i(x,y)\alpha_i\psi(y){\rm d}s(y)\\
      &\ +\int_{\partial D}T^{\kappa_i}_{i,y}[\Gamma_i(x,y)]^{\mathsf T}\beta_i\varphi(y){\rm d}s(y),\qquad x\in  D_i\\
    u(x)=&\int_{\partial D}\Gamma_e(x,y)\alpha_e\psi(y){\rm d}s(y)+\int_{\partial D}T^{\kappa_e}_{e,y}[\Gamma_e(x,y)]^{\mathsf T}\beta_e\varphi(y){\rm d}s(y),\qquad x\in  D_e
\end{split}
\end{equation}
where $\alpha_r,\beta_r$ and $\kappa_r(r=i,e)$ are constants the same as in (\ref{ab}), $\varphi(y),
\psi(y)$ and $\eta(y)$ are the unknown vectors. Taking account into boundary conditions and the properties
of single- and double-layer potentials, we obtain the following system of integral equations
\begin{equation}\label{B}
  \begin{pmatrix}
    -\frac{1}{2}+W_{ii}^{\kappa_i} & \alpha_iS_{ei}^{i} &  -\alpha_iW_{ei}^{\kappa_i,i}\\
    W_{ie}^{\kappa_i} & \alpha_iS^i_{ee}-\alpha_eS_{ee}&  -\frac{1}{2}+(\alpha_eW_{ee}^{\kappa_e}-\alpha_iW_{ee}^{\kappa_i,i})  \\
    V^{\gamma_i,\kappa_i,i}_{ie}&  -\frac{1}{2}+(\alpha_iW_{ee}^{'\gamma_i,i}-\alpha_eW_{ee}^{'\gamma_e}) & \alpha_eV^{\gamma_e,\kappa_e}_{ee}-\alpha_iV^{\gamma_i,\kappa_i,i}_{ee}
  \end{pmatrix}\cdot\begin{pmatrix}
                      \eta \\
                      \psi \\
                      \varphi
                    \end{pmatrix}=\begin{pmatrix}
                                    0 \\
                                    F \\
                                    H
                                  \end{pmatrix}
\end{equation}

In the following, we shall consider the symbolic matrix $\sigma$ of the system (\ref{B1}). Since operators
$ S_{ei}^{i}, W_{ei}^{\kappa_i,i}, W_{ie}^{\kappa_i}$ and $V^{\gamma_i,\kappa_i,i}_{ie}$ are compact with zero symbols. Moreover, $S^i_{ee}$ and $S_{ee}$ are weak singular operators with symbols in  $L^{p}(\partial D), p>1.$  Note that $\sigma$ has the form
\begin{equation}
\nonumber
\sigma=\begin{pmatrix}
           \sigma_{11} & 0 & 0 \\
           0 & 0 &\sigma_{23} \\
           0 & \sigma_{32} & \sigma_{33}
         \end{pmatrix}
\end{equation}
where $\sigma_{kj}$ are also $3\times3$ symbol matrix of the singular operators in (\ref{B}), respectively.
It's obvious that $\sigma_{33}$ is not useful for the determinant of the $\sigma$, so we will only calculate the symbol of $L_1= -\frac{1}{2}+W_{ii}^{\kappa_i},
L_2= -\frac{1}{2}+(\alpha_eW_{ee}^{\kappa_e}-\alpha_iW_{ee}^{\kappa_i,i})$ and
$L_3= -\frac{1}{2}+(\alpha_iW_{ee}^{'\gamma_i,i}-\alpha_eW_{ee}^{'\gamma_e}).$ Then,we first introduce some notations which will be useful.

Kelvin’s matrix $\Gamma_{0i}(x,y)$ and $\Gamma_{0e}(x,y)$ are fundamental solutions to
\begin{align*}
  \mu_i\Delta v(x)+(\lambda_i+\mu_i)\nabla(\nabla\cdot v(x)) =&0, \\
 \mu_e\Delta u(x)+(\lambda_e+\mu_e)\nabla(\nabla\cdot u(x)) =&0.   \\
\end{align*}
respectively. Then denote $\Gamma_{1i}(x,y)$ and $\Gamma_{1e}(x,y)$ as
\begin{align*}
 \Gamma_{1i}(x,y)&=\Gamma_{i}(x,y)-\Gamma_{0i}(x,y),  \\
 \Gamma_{1e}(x,y)&=\Gamma_{e}(x,y)-\Gamma_{0e}(x,y).
\end{align*}
Then by virtue of Theorem II, 1.3 and 1.5 in \cite{Kupradze}, each elements and their first derivatives in matrix $\Gamma_{1i}(x,y)$ and $\Gamma_{1e}(x,y)$ are bounded. Theorem II, 1.3
\cite{Kupradze} shows that second derivatives of the elements of the matrix $\Gamma_{1i}(x,y)$ and $\Gamma_{1e}(x,y)$ with respect to the Cartesian coordinates of $x$ has a singularity of form $l/|x-y|$. Then we know that all single- and double-layer operators with kernels replaced by $\Gamma_{1i}(x,y)$ and $\Gamma_{1e}(x,y)$ respectively are compact operators,
which means they have zero symbols.

For operator $L_2= -\frac{1}{2}+(\alpha_eW_{ee}^{\kappa_e}-\alpha_iW_{ee}^{\kappa_i,i})$, we know that kernels $T^{\kappa_e}_{e,y}[\Gamma_e(x,y)]^{\mathsf T}$, $T^{\kappa_i}_{i,y}[\Gamma_i(x,y)]^{\mathsf T}$ have the same symbols with $T^{\kappa_e}_{e,y}[\Gamma_{0e}(x,y)]^{\mathsf T}, T^{\kappa_i}_{i,y}[\Gamma_{0i}(x,y)]^{\mathsf T}$, respectively. Then
\begin{equation} \label{B2}
  \begin{split}
      (T^{\kappa_e}_{e,y}[\Gamma_{0e}(x,y)]^{\mathsf T})_{mn}= &- \frac{\delta_{mn}}{2\pi}\sum_{l=1}^3\frac{n_l(y)(x_l-y_l)}{|x-y|^3}+\sum_{l=1}^3
  \mathcal{M}_{ml}(\partial_y)\\
   \times & [(\kappa_e\lambda'_e-\mu_e\mu'_e)\frac{\delta_{ln}}{|x-y|}+
   \mu'_e(\kappa_e+\mu_e)\frac{(x_l-y_l)(x_n-y_n)}{|x-y|^3}],\ m,n=1,2,3,
  \end{split}
\end{equation}

where $\mathcal{M}_{ml}(\partial_y)=n_n(y)\frac{\partial}{\partial y_m}-n_m(y)\frac{\partial}{\partial y_n}$, defined the same as in Section
2. Moreover,\\ $\lambda'_e=(\lambda_e+3\mu_e)[4\pi\mu_e(\lambda_e+2\mu_e)]^{-1}$ and
$\mu'_e=(\lambda_e+\mu_e)[4\pi\mu_e(\lambda_e+2\mu_e)]^{-1}$. Since boundary $\partial D$ is $\mathcal{C}^2$ smooth, we have the unit normal
vector function $n(x)\in \mathcal{C}^1(\partial D)$ and
\begin{equation}\label{B1}
  |n(x)-n(y)|=\mathcal{O}(|x-y|), \ \ |\sum_{l=1}^3 n_l(y)(x_l-y_l)|=\mathcal{O}(|x-y|^2).
\end{equation}
by direct calculation, we have according to (\ref{B2}) and  (\ref{B1})
\begin{equation}\label{B3}
  (T^{\kappa_e}_{e,y}[\Gamma_{0e}(x,y)]^{\mathsf T})_{mn}=E_{mn}(x,y)+
  (\kappa_e\lambda'_e-\mu_e\mu'_e) \mathcal{M}_{mn}(\partial_y)\frac{ 1}{|x-y|}
\end{equation}
where $E_{mn}(x,y)$ has a singularity of type $\frac{ 1}{|x-y|}$. So the same procedure applied to $T^{\kappa_i}_{i,y}[\Gamma_{0i}(x,y)]^{\mathsf T}$, we have
\begin{equation}\label{B4}
  (T^{\kappa_e}_{e,y}[\Gamma_{0e}(x,y)]^{\mathsf T})_{mn}=E'_{mn}(x,y)+
  (\kappa_i\lambda'_i-\mu_i\mu'_i) \mathcal{M}_{mn}(\partial_y)\frac{ 1}{|x-y|}
\end{equation}
where $E'_{mn}(x,y)$ has a singularity of type $\frac{ 1}{|x-y|}$. To calculate the symbol of $\mathcal{M}_{mn}(\partial_y)\frac{ 1}{|x-y|}$,
we use a local coordinate system at $y\in \partial D$, then we have the
formula
\begin{equation}\label{B5}
  x_m=y_m+\sum_{l=1}^3 a_{ml}(y)\eta_l
\end{equation}
where $\eta_1, \eta_2, \eta_3$ are the coordinate of $x$ in the system.

Applying transformation (\ref{B5}), we find that the representation
\begin{equation}
\nonumber
\mathcal{M}_{mn}(\partial_y)\frac{ 1}{|x-y|}=\sum_{l=1}^3\varepsilon_{nml}(a_{l2}\eta_1-a_{l1}\eta_2)|\eta|^3
\end{equation}
where $\varepsilon_{nml}$ is the Levi-Civita symbol, $\eta=(\eta_1,\eta_2,0)$, is valid for $x\in \partial D$.

In view of (\ref{B2})-(\ref{B5}), the characteristic of the operator $L_2$ is
\begin{align*}
  \chi_{mm} =&-\frac{1}{2},\ \ m=1,2,3\\
  \chi_{mn} =&(2\pi)^{-1} p\sum_{l=1}^3\varepsilon_{nml}(a_{l2}\frac{\eta_1}{|\eta|}-a_{l1}\frac{\eta_2}{|\eta|})\\
           =& -\chi_{nm},\ m\neq n,
\end{align*}
where $p=(c_eb_i\mu_i-c_ib_e\mu_e)(b_i\mu_i+b_e\mu_e)^{-1}$ with $c_r=\mu_r(\lambda_r+2\mu_r)^{-1}, b_r=(\lambda_r+\mu_r)(\lambda_r+2\mu_r)^{-1}, r=i,e.$

Now the symbol of the operator can be calculated by (\ref{A5})
\begin{align*}
  (\delta_{23})_{mm} =&-\frac{1}{2},\ \ m=1,2,3\\
  (\delta_{23})_{mn} =&i2\pi \chi_{mn} \ \ m\neq n, \ m,n=1,2,3.
\end{align*}
 Then $\mathrm{det} \delta_{23}= p^2-\frac{1}{4}$. Similarly, for operator $L_1, L_3$, we have $\mathrm{det} \delta_{11}=c_i^2-\frac{1}{4}, \mathrm{det} \delta_{32}= q^2-\frac{1}{4}, q=(\mu_i^2-\mu_e^2+(c_i-c_e)\mu_i\mu_e)(\mu_i+\mu_e)^{-2}.$ Finally, we have the determinant of system (\ref{B})
\begin{equation}
 \nonumber
  |\mathrm{det} \delta|=(\frac{1}{4}-c_i^2)(\frac{1}{4}-p^2)(\frac{1}{4}-q^2)>0.
 \end{equation} 
Thus it follows Theorem \ref{th5.1} and Theorem \ref{th5.2}, system (\ref{B}) is normally solvable and have a zero index, which means Fredholm
theorems valid for it.

\end{document}